\documentclass[leqno,twoside, 11pt]{amsart}

\usepackage{latexsym,esint}
\usepackage{psfig}

\theoremstyle{plain}

\def\endproof{\hspace*{\fill}\mbox{\ \rule{.1in}{.1in}}\medskip }

\newtheorem{theorem}{Theorem}[section]

\newtheorem{lemma}[theorem]{Lemma}

\theoremstyle{definition}
\newtheorem{example}[theorem]{Example}
\newtheorem{remark}[theorem]{Remark}

\numberwithin{equation}{section}
\numberwithin{figure}{section}

\begin{document}

\title[korn inequality in thin domains]
{The uniform Korn - Poincar\'e inequality \\ in thin domains\\
\bigskip
L'in\'egalit\'e de Korn - Poincar\'e \\ dans les domaines minces}
\author{Marta Lewicka and Stefan M\"uller}
\address{Marta Lewicka, Department of Mathematics, Rutgers University,
110 Frelinghuysen Rd., Piscataway, NJ 08854-8019, USA }
\address{Stefan M\"uller,  Hausdorff Center for Mathematics
 \& Institute for Applied Mathematics,
Bonn University,
Endenicher Allee 60,
D - 53115 Bonn, Germany}
\email{lewicka@math.rutgers.edu, stefan.mueller@hcm.uni-bonn.de}

\maketitle

\tableofcontents

\begin{abstract} 
We study the Korn-Poincar\'e inequality:
\begin{equation*}
\|u\|_{W^{1,2}(S^h)} \leq C_h \|D(u)\|_{L^2(S^h)},
\end{equation*}
in domains $S^h$ that are shells of small thickness of order $h$, 
around an arbitrary compact, boundaryless and smooth
hypersurface $S$ in $\mathbf{R}^n$.
By $D(u)$ we denote the symmetric part of the gradient $\nabla u$,
and we assume the tangential boundary conditions:
\begin{equation*}
u\cdot\vec n^h = 0 \quad \mbox{ on } \partial S^h.
\end{equation*}
We prove that $C_h$ remains uniformly bounded as $h\to 0$, 
for vector fields $u$ in any family of
cones (with angle $<\pi/2$, uniform in $h$) around the orthogonal
complement of extensions of Killing vector fields on $S$.

We show that this condition is optimal, as in turn every Killing field
admits a family of extensions $u^h$, for which the ratio
$\|u^h\|_{W^{1,2}(S^h)} / \|D(u^h)\|_{L^2(S^h)}$
blows up as $h\to 0$, even if the domains $S^h$ are not rotationally symmetric.
\end{abstract}

\begin{abstract}
On \'etudie l'in\'egalit\'e de Korn-Poincar\'e: 
\begin{equation*}
\|u\|_{W^{1,2}(S^h)} \leq C_h \|D(u)\|_{L^2(S^h)},
\end{equation*}
dans les domaines $S^h$ de type des coques d'\'epaisseurs d'ordre $h$ autour 
d'une hypersurface compacte sans bord et reguli\`ere $S$ de $\mathbf{R}^n$. 
Par $D(u)$, on r\'ef\`ere \`a la partie sym\'etrique du gradient $\nabla u$ 
et on suppose la condition au bord:  
\begin{equation*}
u\cdot\vec n^h = 0 \quad \mbox{ on } \partial S^h.
\end{equation*} 
On d\'emontre que $C_h$ reste uniform\'ement born\'e car $h\to 0$, pour
tout champ de vecteurs dans une famille de c\^ones donn\'ee (faisant un angle    
$< \pi /2$, uniforme en $h$) autour du compl\'ement orthogonal des extensions
de champs de vecteurs de Killing sur $S$.

On montre que cette condition est optimale comme tout champ de Killling $u$ 
sur $S$ admet une famille d'extensions $u^h$ sur $S^h$ pour lesquelles le rapport 
$\|u^h\|_{W^{1,2}(S^h)} / \|D(u^h)\|_{L^2(S^h)}$ 
tend \`a l'infini comme $h\to 0$, m\^eme si les $S^h$ ne poss\`edent pas de 
symmetrie axiale.  
\end{abstract}

\section{Introduction} 

The objective of this paper is to study the Korn-Poincar\'e inequality:
\begin{equation}\label{intro_korn}
\|u\|_{W^{1,2}(S^h)} \leq C_h \|D(u)\|_{L^2(S^h)},
\end{equation}
under the tangential boundary conditions:
\begin{equation}\label{intro_tan}
u\cdot\vec n^h = 0 \quad \mbox{ on } \partial S^h,
\end{equation}
in domains $S^h$ that are shells of small thickness of order $h$, 
around an arbitrary compact, boundaryless and smooth
hypersurface $S$ in $\mathbf{R}^n$.
By $D(u)=\frac{1}{2}(\nabla u + (\nabla u)^T)$ 
we denote the symmetric part of the gradient $\nabla u$.

Korn's inequality was discovered in the early $XX$th century,
in the context of the boundary value problem of linear elastostatics 
\cite{Korn1, Korn2}.
There is by now an extensive literature on the subject, 
relating to various contexts and various boundary conditions 
(see for example a review \cite{Horgan}, and the references therein).
If (\ref{intro_tan}) is replaced by $u=0$ on $\partial S^h$, 
one can easily prove that 
$\|\nabla u\|_{L^2}\leq \sqrt{2}\|D(u)\|_{L^2}$, and so
(\ref{intro_korn}) follows by the Poincar\'e inequality.
In the absence of this boundary condition,
or with its weaker versions, the bound 
(\ref{intro_korn}) requires an extra assumption 
to eliminate pure rotations and translations,
when $D(u)=0$ but $\nabla u\neq 0$. In particular, 
(\ref{intro_korn}) holds for all $W^{1,2}(S^h)$ vector fields $u$
satisfying (\ref{intro_tan}), which are $L^2$- orthogonal
to the space of those linear fields on $S^h$ with skew-symmetric
gradient that are themselves tangent on the boundary.

We are interested in the behaviour of the constant $C_h$,
as $h\to 0$. It turns out that in general, $C_h$ may blow up,
even if $S^h$ are not rotationally symmetric (and so the aforementioned 
spaces are trivial).
The correct way of looking at this problem is to consider the asymptotic
inequality as $h\to 0$, i.e. the related Korn inequality on $S$
(see also \cite{JC}):
\begin{equation}\label{intro_due}
\|v\|_{W^{1,2}(S)} \leq C \|D(v)\|_{L^2(S)}.
\end{equation}
This inequality holds true for all tangent vector fields $v$ on $S$, which are 
$L^2$-orthogonal to the space of Killing fields on $S$.
A Killing field $v$ is defined to be a smooth tangent
vector field which generates a one-parameter family of isometries on $S$.
It is well known that the space of Killing fields on a given surface is a
finite dimensional Lie algebra. 
An equivalent characterisation is:
\begin{equation}\label{intro_tre}
D(v) = 0, \qquad \mbox{ i.e.: } \tau\nabla v(x) \tau  =0 
\quad \forall x\in S \quad \forall \tau \in T_x S.
\end{equation}

In this paper, we first notice that any $v$ satisfying (\ref{intro_tre})
admits a family of
extensions $v^h : S^h \to \mathbf{R}^n$, such that the boundary conditions
(\ref{intro_tan}) hold and so that the ratio
$\|v^h\|_{W^{1,2}(S^h)} / \|D(v^h)\|_{L^2 (S^h)}$ goes to infinity
as $h\to 0$.  
This construction turns out to be 
the worst case scenario for the possible blow-up of $C_h$.
Our main results state that the constants $C_h$
remain uniformly bounded for vector fields $u$ inside any family of
cones (with angle $<\pi/2$, uniform in $h$) around the orthogonal
complement of the space of extensions of all Killing fields on $S$.

\medskip

Our main motivation in this work has been its application
to dynamics of Navier-Stokes equations in thin $3$-dimensional domains.
Thin domains are encountered in many problems in solid or fluid mechanics.
For example, in ocean dynamics, one is dealing with
the fluid regions which are thin compared to the horizontal length scales.
Other examples include lubrication, meteorology, blood circulation etc.;
they are a part of a broader study of the behaviour of
various PDEs on thin $n$-dimensional domains, where $n\geq2$ (for
a review see \cite{Ra95}).

The study of the global existence and asymptotic properties of
solutions to the Navier-Stokes equations
in thin $3$d domains began with Raugel and Sell in \cite{RS93}.
They proved global existence of strong solutions for large initial
data and in presence of large forcing, for the sufficiently 
thin $3$d product domain $\Omega = Q \times (0, \epsilon)$, 
with the boundary conditions
either purely periodical or combined periodic-Dirichlet.
Further generalisations to other boundary conditions 
followed (see the references in \cite{irs}).
Towards analysing thin domains other than simple product domains,
Iftimie, Raugel and Sell \cite{irs} treated domains of the type:
$ \Omega = \{x\in {\mathbf R}^3; ~(x_1, x_2)\in Q,~
0 < x_3 < \epsilon g (x_1, x_2)\}$,
with the mixed boundary conditions:  periodic on the lateral boundary
and the Navier boundary conditions:
\begin{equation}\label{intro_navier}
D(u) \vec n^h || \vec n^h  \mbox{ and } u\cdot \vec n^h =0 \quad 
\mbox{ on } \partial S^h
\end{equation}
on the top and on the bottom.


The Korn inequality arises naturally when one considers the incompressible
flow subject to (\ref{intro_navier}), for the following reason. 
In order to define the relevant Stokes operator
one uses the symmetric bilinear form $B(u,v)= \int D(u) : D(v)$
rather than the usual $\int \nabla u : \nabla v$. 
Hence the energy methods give suitable bounds for $\|D(u^h)\|_{L^2(S^h)}$,
for a solution flow $u^h$ in $S^h$. 
On the other hand, in order to establish compactness in the
limit problem as $h \to 0$, one needs bounds for
the $W^{1,2}$ norm of $u^h$, with constants independent of $h$. 
The inequality (\ref{intro_korn}) (with uniform constants $C_h$)  
provides thus a necessary uniform equivalence of the two norms
$\|u^h\|_{W^{1,2}}$ and $\|D(u^h)\|_{L^2}$ on $S^h$.

It is therefore hoped that we can apply the result of this paper 
to study the dynamics of the Navier-Stokes equations, 
under the Navier boundary conditions,
in thin shells with various geometries of the reference surface $S$ 
and of the boundaries of $S^h$.

\medskip

Starting with the original papers of Korn \cite{Korn1, Korn2}, 
Korn's inequality has also 
been widely used as a basic tool for the existence of solutions of
the linearised  dis\-pla\-ce\-ment-traction equations in elasticity
\cite{F, ciarbookvol1, Horgan}.
In this context, for a given displacement vector field $u$, the matrix
field $D(u)$ is the linearised strain, which measures the pointwise
deviation of the deformation $Id +\epsilon u$  from a rigid motion,
up to the first order terms in $\epsilon$.
Hence, Korn's inequality can be interpreted as a rigidity estimate for
small displacement deformations: they are $W^{1,2}$ close to $Id$, by the error
given in the right hand side of (\ref{intro_korn}).  
A nonlinear version of this rigidity estimate, obtained recently 
in \cite{FJMgeo}, has been extensively applied to problems in nonlinear
elasticity and plate theories (see eg \cite{FJMgeo, FJMhier}).
Earlier, Korn's inequalities in thin neighbourhoods of flat surfaces
have been discussed  in series of papers by Kohn and Vogelius
\cite{KV2}. They derive an estimate which degenerates as $h \to 0$
for clamped boundary conditions at the side of the plate. An analogous
result in our setting is given in Theorem \ref{th_very_weak}. 

\bigskip

\noindent{\bf Acknowledgments.} 
M.L. was partially supported by the NSF grants DMS-0707275 and DMS-0846996,
and by the Polish MN grant N N201 547438.

\section{The main theorems}

Let $S$ be a smooth, closed hypersurface (a compact boundaryless manifold 
of co-dimension $1$) in $\mathbf{R}^n$.
Consider a family $\{S^h\}_{h>0}$ of thin shells around $S$:
$$S^h = \{z=x + t\vec n(x); ~ x\in S, ~ -g_1^h(x) < t < g_2^h(x)\},$$
whose boundary is given by smooth positive functions 
$g_1^h, g_2^h: S\longrightarrow \mathbf{R}$.
We will use the following notation:
$\vec n^h$ for the outward unit normal to $\partial S^h$, $\vec n(x)$
for the outward unit normal to $S$ (seen as the boundary of some bounded domain in 
$\mathbf{R}^n$), $T_x S$ for the tangent space to $S$ at a given $x\in S$.
The projection onto $S$ along $\vec n$ will be denoted by $\pi$, so that,
for $h$ sufficiently small:
$$\pi(z) = x \qquad \forall z=x+t\vec n(x)\in S^h.$$

The standard Korn inequality (see Theorem \ref{korn_standard} in Appendix A)
on bounded Lipschitz domains implies that for each
$u\in W^{1,2}(S^h, \mathbf{R}^n)$ satisfying the orthogonality condition:
\begin{equation}\label{ort_lin}
\int_{S^h} u(z)\cdot (Az+b) ~\mbox{d}z = 0, \quad \forall A\in so(n), 
\quad\forall b\in\mathbf{R}^n
\end{equation} 
one has:
\begin{equation}\label{korn_nonunif}
\|u\|_{W^{1,2}(S^h)} \leq C_h \|D(u)\|_{L^2(S^h)}
\end{equation}
and the constant $C_h$ depends only on the domain $S^h$, but not on $u$.
Here, $so(n)$ stands for the linear space of all $n\times n$ skew-symmetric matrices:
$$so(n) = \{A\in M^{n\times n}; ~ A = -A^T\} = 
\{A\in M^{n\times n}; ~ \tau^T A\tau=0 \quad \forall\tau\in\mathbf{R}^n\}$$
while by $D(u)$ we mean the symmetric part of $\nabla u$:
$$D(u) = \frac{1}{2} \left(\nabla u + (\nabla u)^T\right).$$

The same result is true for $u$ satisfying additionally:
$$u\cdot \vec n^h=0 ~~~\mbox{ on }\partial S^h,$$
when in (\ref{ort_lin}) we take only linear fields
$Az+b\in\mathcal{R}_\partial(S^h)$;
with skew-symmetric gradient, and satisfying the same boundary condition as $u$:
$$\mathcal{R}_\partial(S^h)=\{w=Az+b; ~~ A\in so(n), b\in\mathbf{R}^n, w\cdot\vec n^h=0
\mbox{ on } \partial S^h\}.$$
The standard proof by contradiction (see Theorem \ref{naive_korn} in Appendix A) 
shows that the constant $C_h$ in (\ref{korn_nonunif}) again 
does not depend on $u$  but it may depend on the geometry
of $S^h$. In particular, as follows from the example in section \ref{example}, 
$C_h$ may converge to infinity 
as the thickness of $S^h$ (that is $\|g_1^h + g_2^h\|_{L^\infty(S)}$) 
converges to $0$.
Our goal is to investigate the behaviour of $C_h$ in two frameworks, relating to the following 
hypotheses:
\begin{itemize}
\item[{\bf (H1)}] For some positive constants $C_1, C_2$ and 
$C_3$, and all small $h>0$ there holds:
$$C_1 h\leq g_i^h(x)\leq C_2 h, \qquad |\nabla g_i^h(x)|\leq C_3 h \qquad  
\forall x\in S, \quad i=1,2. $$ 
\item[{\bf (H2)}] For some smooth positive functions $g_1, g_2:S\longrightarrow \mathbf{R}$, 
there holds:
$$\frac{1}{h} g_i^h \to g_i \quad {\mbox{ in }} \mathcal{C}^1(S) \quad {\mbox{ as }} h\to 0,
\qquad  i=1,2. $$ 
\end{itemize}
Clearly {\bf (H2)} implies {\bf (H1)} with:
$C_1 = 1/2\min \{g_i(x); ~x\in S, ~i=1,2\}$,
$C_2 = 2\max_i\|g_i\|_{L^\infty}$,
$C_3 = \max_i \|\nabla g_i\|_{L^\infty (S)} +1$.  

\medskip

Before stating our main results, we need to recall the notion of 
a Killing vector field. 
The Lie algebra of smooth Killing fields on $S$ will be denoted 
by $\mathcal{I}(S)$. 
That is, $v\in \mathcal{I}(S)$ if and only if:
\begin{itemize}
\item[(i)] $v:S\longrightarrow \mathbf{R}^n$ is smooth and $v(x)\in T_x(S)$ 
for every $x\in S$,
\item[(ii)] $\begin{displaystyle}\frac{\partial v}{\partial \tau} (x) \cdot \tau = 0
\end{displaystyle}$ for every $x\in S$ and every $\tau\in T_x S$.
\end{itemize}
Here $\partial v/\partial\tau(x)$ denotes the derivative of $v$ in the tangent 
direction $\tau$, i.e. if $\gamma: (-\epsilon,\epsilon) \longrightarrow S$ 
is a $\mathcal{C}^1$ curve with $\gamma(0) = x$ and $\gamma'(0) = \tau$,
then $\partial v/\partial \tau(x) = (v\circ \gamma)'(0)$.
Condition (ii) implies that
\begin{equation}\label{kill}
\frac{\partial v}{\partial \tau}(x) \cdot \eta 
+ \frac{\partial v}{\partial \eta}(x)\cdot \tau = 0 \qquad \forall \tau,\eta\in T_xS
\quad \forall x\in S.
\end{equation}
Recall that Killing vector fields are infinitesimal generators of isometries on $S$, 
in the sense that if $\Phi$ is the flow generated by $v$:
$$ \frac{\mbox{d}}{\mbox{d}s} \Phi (s,x) = v(\Phi(s,x)), \qquad
\Phi(0,x) = x,$$
then for every fixed $s$ the map $S\ni x\mapsto \Phi(s,x)\in S$ is an isometry.
The linear space $\mathcal{I}(S)$ has finite dimension 
\cite{KN, Peterson}. Also, any Killing field of class $W^{1,2}$ is in fact smooth
(see Lemma \ref{weak_kil}); we recall these facts in Appendix C.

For $g_1, g_2:S\longrightarrow \mathbf{R}$, define the subspace of $\mathcal{I}(S)$:
$$\mathcal{I}_{g_1, g_2}(S) = \left\{v\in\mathcal{I}(S); 
~~ v(x)\cdot\nabla(g_1+g_2)(x) = 0
\quad \mbox{ for all } x\in S\right\},$$
formed of those Killing fields $v$ which satisfy 
$\lim_{h\to 0}h^{-1}v\cdot (\vec n^h_+ + \vec n^h_-) = 0$, where $\vec n^h_+$ and 
$\vec n^h_-$ denote, respectively, the outward normals to $S^h$ at its boundary points 
$x+ g_2^h(x)$ and $x-g_1^h(x)$.
 
\medskip

Our main results are the following:

\begin{theorem}\label{th1}
Assume {\bf (H1)} and let $\alpha\in [0,1)$.
Then, for all $h>0$ sufficiently small and all $u\in W^{1,2}(S^h, \mathbf{R}^n)$
satisfying one of the following tangency conditions:
$$ u\cdot \vec n^h=0 ~~~ \mbox{ on } 
\partial^+ S^h=\{x+g_2^h(x)\vec n(x); ~x\in S\},$$
or:
$$ u\cdot \vec n^h=0 ~~~ \mbox{ on } 
\partial^- S^h=\{x-g_1^h(x)\vec n(x); ~x\in S\},$$
together with:
\begin{equation}\label{f1}
\left|\int_{S^h} u(z) v(\pi(z)) ~\mathrm{d}z\right| 
\leq \alpha \|u\|_{L^2(S^h)} \cdot \|v\pi\|_{L^2(S^h)}
\qquad \forall v\in\mathcal{I}(S),
\end{equation}
there holds:
\begin{equation}\label{korn}
\|u\|_{W^{1,2}(S^h)} \leq C \|D(u)\|_{L^2(S^h)},
\end{equation}
where $C$ is independent of $u$ and of $h$.
\end{theorem}

\begin{theorem}\label{th2}
Assume {\bf (H2)} and let $\alpha\in [0,1)$.
Then for all $h>0$ sufficiently small and all $u\in W^{1,2}(S^h, \mathbf{R}^n)$
satisfying $ u\cdot \vec n^h=0$ on $\partial S^h$
and:
\begin{equation}\label{f2}
\left|\int_{S^h} u(z) v(\pi(z)) ~\mathrm{d}z\right| 
\leq \alpha \|u\|_{L^2(S^h)} \cdot \|v\pi\|_{L^2(S^h)}
\qquad \forall v\in\mathcal{I}_{g_1, g_2}(S),
\end{equation}
there holds (\ref{korn})
with $C$ independent of $u$ and of $h$.
\end{theorem}

The example constructed in section \ref{example} shows that conditions
(\ref{f1}) (or (\ref{f2})) are necessary for the bound (\ref{korn}). In particular,
any Killing field $v$ on $S$ generates a family of vector fields $v^h$ on $S^h$,
satisfying the boundary condition and such that
$\|\nabla v^h\|_{L^2(S^h)}^2 \geq Ch$ but $\|D(v^h)\|_{L^2(S^h)}^2\leq Ch^3$.
Hence, if one naively assumes that $u$ satisfies the angle condition 
only with the space of generators of appropriate rotations on $S$, rather 
than the whole $\mathcal{I}(S)$, the constant $C_h$ has a blow-up rate of at least 
$h^{-1}$, as $h\to 0$. The following theorem shows that this is the actual
blow-up rate, under the above mentioned conditions. 

More precisely, define:
$$\mathcal{R}(S) = \big\{v:S\longrightarrow\mathbf{R}^n; ~~ v(x) = Ax+b, ~
A\in so(n), ~b\in\mathbf{R}^n, ~ v\cdot \vec n=0 \mbox{ on } S\big\}
\subset \mathcal{I(S)},$$
$$\mathcal{R}_{g_1, g_2}(S) = \big\{v\in\mathcal{R}(S); 
~~ v(x)\cdot \nabla(g_1 + g_2)(x) = 0
\mbox{ for all } x\in S\big\}\subset \mathcal{I}_{g_1, g_2}(S).$$

\begin{theorem}\label{th_very_weak}
Let $\alpha\in [0,1)$.
Then, for all $h$ sufficiently small and all $u\in W^{1,2}(S^h, \mathbf{R}^n)$, 
there holds:
\begin{equation}\label{very_weak}
\|u\|_{W^{1,2}(S^h)} \leq Ch^{-1} \|D(u)\|_{L^2(S^h)},
\end{equation}
in any of the following situations:
\begin{itemize}
\item[(i)] {\bf (H1)} holds, $u\cdot\vec n^h=0$ on $\partial^+ S^h$
or $u\cdot\vec n^h=0$ on $\partial^- S^h$,  and:
$$\left|\int_{S^h} u(z) v(\pi(z)) ~\mathrm{d}z\right| 
\leq \alpha \|u\|_{L^2(S^h)} \cdot \|v\pi\|_{L^2(S^h)}
\qquad \forall v\in\mathcal{R}(S).$$
\item[(ii)]  {\bf (H2)} holds, $u\cdot\vec n^h=0$ on $\partial S^h$,   and:
$$\left|\int_{S^h} u(z) v(\pi(z)) ~\mathrm{d}z\right| 
\leq \alpha \|u\|_{L^2(S^h)} \cdot \|v\pi\|_{L^2(S^h)}
\qquad \forall v\in\mathcal{R}_{g_1, g_2}(S).$$
\end{itemize}
\end{theorem}
\noindent Notice that (i) is implied by the hypotheses of Theorem \ref{th1} and 
(ii) by the hypotheses of Theorem \ref{th2}, as the spaces $\mathcal{R}(S)$
and $\mathcal{R}_{g_1, g_2}(S)$
are contained in $\mathcal{I}(S)$ or $\mathcal{I}_{g_1,g_2}(S)$,
respectively. The bound (\ref{very_weak}) was obtained also in \cite{KV2}, 
but in a different context of thin plates with clamped boundary conditions 
and rapidly varying thickness.

\section{Remarks and an outline of proofs}

\begin{remark} 
Conditions (\ref{f1}) and (\ref{f2}) may be understood in the following way: 
the cosine of the angle (in $L^2(S^h)$) between $u$ and its projection onto 
the linear space $W^h\subset L^2(S^h)$ of `trivial' extensions 
$v\pi$ of certain Killing fields 
$v\in\mathcal{I}(S)$ (or $v\in\mathcal{I}_{g_1, g_2}(S)$) should be smaller than
$\alpha$. 

Equivalently,  one considers vector fields $u\in W^{1,2}(S^h)$, which
for a given constant $\beta\geq 1$ (related to $\alpha$ through:
$\beta = (1-\alpha^2)^{-1/2}$) satisfy:
\begin{equation}\label{dist_angle}
\|u\|_{L^2(S^h)} \leq \beta \|u - v\pi\|_{L^2(S^h)} \qquad 
\forall v\in\mathcal{I}(S) \quad \mbox{ (or } \forall v\in\mathcal{I}_{g_1, g_2}(S)
\mbox{)}.
\end{equation}
That is, the distance of $u$ from the space $W^h$ controls (uniformly) the full 
norm $\|u\|_{L^2(S^h)}$.

By Theorems \ref{th1} and \ref{th2}, inside each closed 
cone around $(W^h)^\perp$, of fixed angle $\theta<\pi/2$ in $L^2(S^h)$, the bound 
(\ref{korn}) holds, with a constant $C$, that is uniform in $u$ and $h$.
One could therefore think that $W^h$ is the kernel for the uniform 
Korn-Poincar\'e
inequality, in the same manner as the linear maps $Az+b$ with skew gradients
$A\in so(n)$ constitute the kernel for the standard Korn inequality 
(\ref{ort_lin}), (\ref{korn_nonunif}). 
This is not exactly the case, as the uniform Korn inequality is
true for the extensions $v\pi$ (see Remark \ref{rem_trivial_ext}).
The role of the kernel is played by the space $\widetilde{W}^h$
of 'smart' extensions $v^h$ of the Killing fields $v$ 
(see the formula (\ref{3_def})).

Still, with  $v\pi$ replaced by $v^h$ in (\ref{f1}) or (\ref{f2}), 
both Theorems \ref{th1} and \ref{th2} remain true.
This is because the spaces $W^h$ and $\widetilde{W}^h$ are asymptotically
tangent at $h=0$:
$$\|v\pi - v^h\|_{L^2(S^h)}\leq Ch \|v\pi\|_{L^2(S^h)} \qquad 
\forall v\in\mathcal{I}(S).$$
Hence, if $|\langle u,v^h\rangle_{L^2}|\leq \alpha 
\|u\|_{L^2}\cdot \|v^h\|_{L^2}$ for some $\alpha<1$,
then $|\langle u,v\pi\rangle_{L^2}|\leq (\alpha+Ch) 
\|u\|_{L^2}\cdot \|v\pi\|_{L^2}$, and the angle conditions in main theorems hold, 
for $h$ sufficiently small.
The fact that we chose to work with 'trivial' extensions, in $W^h$
(giving a simpler condition), instead of the real kernel $\widetilde{W}^h$,
is thus not restrictive.

In the particular case when $\partial S^h$ is parallel to $S$, 
say $g_i^h = h$, we have
$$\vec n^h(x + g_2(x)\vec n(x)) = \vec n(x), 
\qquad \vec n^h(x - g_1(x)\vec n(x)) = -\vec n(x),$$
$$\mathcal{I}(S) = \mathcal{I}_{g_1, g_2}(S).$$
If $w\in\mathcal{R}_\partial (S^h)$ then $w_{|S}$ is tangent to $S$ and, 
as shown in
Appendix A (Theorem \ref{lemma_rotations}) it generates a rotation on $S$.
Actually: $w = (w_{\mid S})^h\in\widetilde{W}^h$ and so 
by the preceding comment we see that 
the condition  (\ref{ort_lin}) is asymptotically contained 
in  (\ref{f1}) (or (\ref{f2})).
\end{remark}

\begin{remark}
A natural question is whether $\mathcal{I}(S)$ may contain other vector fields 
than the restrictions of generators of rigid motions on the whole $\mathbf{R}^n$.
This is clearly the case when $n=2$: any tangent vector field of constant length 
is a Killing field.
The same question for higher dimensions and even for $n=3$ and general (nonconvex)
hypersurfaces is open, to our knowledge. 
It is closely related to other open problems: whether the class of rotationally 
symmetric surfaces is closed under intrinsic isometries; or whether 
every intrinsic isometry on $S$ is actually a restriction of some 
isometry of $\mathbf{R}^3$. When $S$ is convex, it is well known that the last property 
holds, while for non-convex surfaces it does not. The answer to the same question,
formulated for $1$-parameter families of isometries is not known
(see \cite{Spivak} vol. 5).
\end{remark}

\bigskip

\noindent {\bf An outline of proofs of Theorems \ref{th1} and \ref{th2}.}

The general strategy is as follows. Suppose that $\|D(u)\|_{L^2(S^h)}$ is small.
It is natural to study the map $\bar{u}: S\longrightarrow\mathbf{R}^n$
which is obtained by averaging $u$ in the normal direction:
$$\bar{u}(x)  = \fint_{-g_1^h(x)}^{g_2^h(x)} u(x+t\vec n(x))~\mbox{d}t$$
(see e.g. \cite{Ra95, RS93, irs, Grisorods, Grisonone, Grisonew}).
By the boundary condition, one has $\bar{u}\cdot\vec n\approx 0$, i.e. $\bar{u}$ is
almost tangential to $S$. Moreover, $D(\bar{u})$ is essentially bounded
by the average of $D(u)$. 
Hence if $D(u)$ is small, by Korn's inequality on surfaces,
the field $\bar{u}$ must be close to a Killing field $v$. If $v$ is not small,
we will get a contradiction to the angle condition (\ref{f1}) or (\ref{f2}).
If $v$ is small then we get good estimates for $\bar{u}$ and ultimately for $u$.

More precisely, the proof proceeds as follows. First (see 
Theorem~\ref{approx}),
an application of Korn's inequality to cylinders of size $h$ and an
 interpolation
argument yield a smooth field $R:S\longrightarrow so(n)$ such that:
\begin{eqnarray}
\int_{S^h}|\nabla u -R\pi|^2 \leq C\int_{S^h}|D(u)|^2, \label{a}\\
\int_{S}|\nabla R|^2 \leq Ch^{-3}\int_{S^h}|D(u)|^2. \label{b}
\end{eqnarray}
From this we deduce (see Lemma \ref{lem2}):
\begin{equation}\label{c}
\int_{S}|\nabla \bar{u} - R_{tan}|^2 \leq Ch^{-1}\int_{S^h}|D(u)|^2
+ Ch \int_{S^h}|\nabla u|^2,
\end{equation}
where $R_{tan}\tau = R\tau$ for all tangent fields $\tau$ and $R_{tan}\vec n=0$.

Using the boundary conditions it is easy to show that (see Lemma \ref{lem1}):
\begin{equation}\label{d}
\int_{S}|\bar{u}\cdot\vec n|^2 \leq Ch \int_{S^h}|\nabla u|^2.
\end{equation}
It is thus natural to study the tangent field:
$$\bar{u}_{tan} = \bar{u}-(\bar{u}\cdot\vec n)\vec n.$$
Now Korn's inequality on $S$ implies that there exists a Killing field $v$ 
such that:
$$\|\bar{u}_{tan} - v\|_{W^{1,2}(S)} \leq C\|D(\bar{u}_{tan})\|_{L^2(S)}.$$
By the angle condition, $v$ must be small in $L^2(S)$,
and hence in $W^{1,2}(S)$ since the Killing fields form a finite dimensional 
space. Thus, $\|\bar{u}_{tan}\|_{W^{1,2}(S)}$ is controlled,
and by (\ref{d}) $\|\bar{u}\|_{L^2(S)}$ is also controlled.
Now the crucial step is to combine (\ref{b}) and (\ref{c}) to deduce that:
\begin{equation}\label{e}
\int_{S}|\nabla (\bar{u}\cdot\vec n)|^2 + |R\vec n|^2 \leq Ch^{-3/2}
\|D(u)\|_{L^2(S^h)}\cdot \|\bar{u}\cdot\vec n\|_{L^2(S)}
+ \mbox{ harmless terms }
\end{equation}
(see Lemma \ref{lem3}). From (\ref{e}) and (\ref{d}) we obtain control on 
$\nabla\bar{u}$. By (\ref{c}) this controls $R_{tan}$, hence $R$, and finally 
(\ref{a}) gives the estimate for $\nabla u$.
The actual argument is by contradiction, assuming that
$h^{-1/2}\|u^h\|_{W^{1,2}(S^h)}=1$ 
and $h^{-1/2}\|D(u^h)\|_{L^h(S^h)}\longrightarrow 0$ 
(see section \ref{section_endproof}).

\medskip

Above and in all subsequent proofs, $C$ denotes an arbitrary positive constant, 
depending on the geometry of $S$ and constants $C_1, C_2, C_3$  
in {\bf (H1)} or the functions $g_1, g_2$ in {\bf (H2)}. 
The constant $C$ may also depend on the choice of $\alpha$, but it 
is always independent of $u$ and $h$.

\section{An example where the constant $C_h$ blows up}\label{example}

Let $g_1, g_2:S\longrightarrow \mathbf{R}$ be some positive and smooth functions, 
and let $g_i^h = h g_i$, $i=1,2$.
Assume that on $S$ there exists a nonzero Killing vector field $v$ such that:
\begin{equation}\label{ex_uno}
v\in \mathcal{I}_{g_1, g_2}(S).
\end{equation}
We are going to construct a family $v^h\in W^{1,2}(S^h, \mathbf{R}^n)$
satisfying the boundary condition
\begin{equation}\label{ex_bd}
v^h\cdot \vec n^h=0 \qquad \mbox{on } \partial S^h,
\end{equation}
for which the uniform bound (\ref{korn}) is not valid (after we take $u^h=v^h$).

\medskip

By $\Pi(x) = \nabla \vec n(x)$ we denote the shape operator on $S$,
that is, the (tangential) gradient of $\vec{n}$. 
For all $x\in S$ and all $t\in (-h g_1(x), h g_2(x))$ define:
\begin{equation}\label{3_def}
v^h(x+t\vec n(x)) = \Big( \mbox{Id} + t \Pi(x) + h \vec n(x) \otimes 
\nabla g_2(x)\Big) v(x). 
\end{equation}
By (\ref{ex_uno}) we obtain:
\begin{equation*}
\begin{split}
v^h(x+t\vec n(x)) = &\frac{hg_1(x) + t}{h(g_1(x) + g_2(x))}\cdot 
\Big( \mbox{Id} + h g_2(x) \Pi(x) + h \vec n(x) \otimes \nabla g_2(x)\Big) v(x)\\
&+ \frac{hg_2(x) - t}{h(g_1(x) + g_2(x))}\cdot 
\Big( \mbox{Id} - h g_1(x) \Pi(x) - h \vec n(x) \otimes \nabla g_1(x)\Big) v(x),
\end{split}
\end{equation*}
which means that each $v^h$ is a linear interpolation between 
the push-forward of the vector field
$v$ from $S$ onto the external part $\partial^+ S^h$ 
of the boundary of $S^h$ and the other push-forward
onto the internal part  $\partial^- S^h$ of $\partial S^h$ 
(see figure \ref{ex_fig}.1). 
Indeed, the derivative of the map:

$$S\ni x\mapsto x\pm hg_i(x) \vec n(x)$$
is given through:
$$ \mbox{Id} \pm hg_i(x) \Pi(x) \pm h \vec n(x) \otimes\nabla g_i(x).$$
In particular, we see that (\ref{ex_bd}) holds.
\begin{figure}[h] \label{ex_fig}
\centerline{\psfig{figure=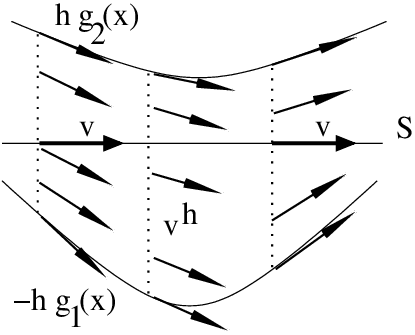,width=6cm,angle=0}}
\caption{The vector fields $v^h$ and $v$.}
\end{figure}
\medskip

Write now $v^h = w + (v^h - w)$, with:
$$w(z) = \Big( \mbox{Id} + t \Pi(x)\Big) v(x), \qquad z = x+t\vec n(x).$$
We wish to estimate the order of different coefficients in $\nabla w$ and $D(w)$.
For every $\tau\in T_xS$, $x\in S$, there holds:
\begin{equation}\label{ex_calc_1}
\begin{split}
\frac{\partial w}{\partial \vec n}(z) & = \Pi(x)v(x),\\
\frac{\partial w}{\partial \tau}(z) & =
t\frac{\partial \Pi}{\partial ((\mbox{Id} + t\Pi(x))^{-1}\tau)}(x) v(x) 
+ (\mbox{Id} + t \Pi(x)) \nabla v(x) (\mbox{Id} + t \Pi(x))^{-1}\tau.
\end{split}
\end{equation}
Observe that:
\begin{equation}\label{ex_calc_2}
\begin{split}
& \left(\frac{\partial w}{\partial \tau}\cdot \vec n + 
\frac{\partial w}{\partial \vec n}\cdot \tau\right) (z) = 
\left(- \frac{\partial \vec n}{\partial \tau}\cdot w + 
\frac{\partial w}{\partial \vec n}\cdot \tau\right) (z)\\
& \qquad \qquad = -\big( \Pi(x)(\mbox{Id} + t \Pi(x))^{-1}\tau\big)\cdot
(\mbox{Id} + t \Pi(x)) v(x) +  \Pi(x)v(x)\cdot \tau \\
& \qquad \qquad = 0,
\end{split}
\end{equation}
because $\vec n\cdot w = 0$ and the symmetric form $\Pi(x)$ commutes with 
$(\mbox{Id} + t \Pi(x))^{-1}$. Likewise:
\begin{equation}\label{ex_calc_3}
\left(\frac{\partial w}{\partial \vec n}\cdot \vec n\right)  (z) = 0.
\end{equation}
To estimate $\eta^T D(w)(z)\tau$, for $\tau, \eta\in T_xS$, notice that:
\begin{equation*}
\begin{split}
\big|\eta^T (\mbox{Id} &+ t \Pi(x)) \nabla v(x) (\mbox{Id} + t \Pi(x))^{-1}\tau\\
& \qquad - \eta^T (\mbox{Id} + t \Pi(x))^{-1} \nabla v(x) 
 (\mbox{Id} + t \Pi(x))^{-1}\tau\big| \leq C t |\nabla v(x)|,
\end{split}
\end{equation*}
because $|(\mbox{Id} + t \Pi(x)) - (\mbox{Id} + t \Pi(x))^{-1}|\leq Ct$.
Above and in the sequel, $C$ denotes any positive constant independent of $h$.
Since $\tau (\mbox{Id} + t \Pi(x))^{-1}\in T_xS$, by (\ref{kill}) 
and (\ref{ex_calc_1}) we obtain:
\begin{equation}\label{ex_calc_4}
|\eta^T D(w)(z)\tau|\leq Ct (|v(x)| + |\nabla v(x)|).
\end{equation}
We also have: $|\nabla (v^h - w)(z)|\leq Ch$ and
by (\ref{ex_calc_2}),  (\ref{ex_calc_3}) and (\ref{ex_calc_4}):
$|D(w) (z)|\leq Ch$ for every $z\in S^h$.  Hence:
$$\|D(v^h)\|^2_{L^2(S^h)} \leq C h^3.$$
On the other hand, inspecting the terms in (\ref{ex_calc_1}) 
and recalling that $v\neq 0$
(and therefore $\nabla v \neq 0$ as well) we see that:
$$\|\nabla v^h\|^2_{L^2(S^h)} \geq \frac{1}{2} \|\nabla v\|^2_{L^2(S^h)} - C h^3 
\geq C h.$$
The two last inequalities imply that the uniform bound (\ref{korn}) is not valid, 
without the restriction (\ref{f2}). Even if $S$ has no rotational symmetry, 
the constants $C_h$ in (\ref{korn_nonunif}) become unbounded as $h\to 0$. 

\begin{remark}\label{rem_trivial_ext}
The construction (\ref{3_def}) is crucial for the counterexample to work.
Indeed, one cannot simply take 'trivial' extensions
$v\pi\in W^{1,2}(S^h)$ for the blow-up of $C_h$.
The reason is that, for any $\tau\in T_x S$, one has:
\begin{equation*}
\begin{split}
& \frac{\partial (v\pi)}{\partial \tau}(z) \cdot \vec n =
- \frac{\partial (\vec n\pi)}{\partial \tau}(z) \cdot (v\pi)(z) 
= - \Pi(x) (\mbox{Id} + t\Pi(x))^{-1}\tau \cdot v(x) 
= \mathcal{O} (1),\\
& \frac{\partial (v\pi)}{\partial \vec n}(z)  = 0,
\end{split}
\end{equation*}
and thus both $\nabla(v\pi)(z)$ and $D(v\pi)(z)$ are of the order 
$\mathcal{O} (1)$. Hence, with a uniform constant $C$:
$$\|\nabla(v\pi)\|_{L^2(S^h)}^2 \leq Ch \|v\|_{W^{1,2}(S)}^2
\leq Ch\leq Ch \|v\|_{L^2(S)}^2 \leq C \|D(v\pi)\|_{L^2(S^h)}^2.$$
\end{remark}

\section{An approximation of $\nabla u$}
\label{section_beginproof}

In this section we construct a smooth function $R$ with skew-symmetric
matrix values, approximating $\nabla u$ on $S^h$ with the error $\|D(u)\|_{L^2(S^h)}$.
The construction relies on Appendix B, where for convenience of the reader 
we analyse the constant in Korn's inequality on a fixed, 
star-shaped with respect to a ball domain (Theorem \ref{app_main_korn}).
We apply this estimate locally and then use a mollification argument 
as in \cite{FJMgeo}. The same approximation result is independently obtained
in \cite{Grisonew} Theorem 4.3, in the context of the unfolding method in 
the linearized elasticity.

As always, $C$ denotes any uniform constant, independent of $u$ and $h$.

\begin{theorem}\label{approx}
Assume {\bf (H1)}. For every $u\in W^{1,2}(S^h,\mathbf{R}^n)$ 
there exists a smooth map $R:S\longrightarrow so(n)$ such that:
\begin{itemize}
\item[(i)] $\begin{displaystyle}
\|\nabla u - R\pi\|_{L^2(S^h)} \leq C  \|D(u)\|_{L^2(S^h)},
\end{displaystyle}$
\item[(ii)]  $\begin{displaystyle}
\|\nabla R \|_{L^2(S)}  \leq C h^{-3/2} \|D(u)\|_{L^2(S^h)}.
\end{displaystyle}$
\end{itemize}
\end{theorem}
\begin{proof}
{\bf 1.} 
For $x\in S$ consider balls in $S$ and 'cylinders' in $S^h$ defined by:
$$D_{x,h} = B(x,h)\cap S, \qquad B_{x,h} = \pi^{-1}(D_{x,h})\cap S^h.$$
The main observation is that sets $B_{x,h}$ are contained in a ball of radius
$(C_2+1)h$ and are star-shaped with respect to a ball of radius
$r(C_1, C_2, C_3,S)h$, when $h$ is sufficiently small. 
Hence, an application of Korn's inequality on $B_{x,h}$
(see Theorem \ref{app_main_korn}) yields a skew-symmetric matrix $A_{x,h}\in so(n)$
such that:
\begin{equation}\label{1.uno}
\int_{B_{x,h}} |\nabla u(z) - A_{x,h}|^2~\mbox{d}z \leq C \int_{B_{x,h}} |D(u)|^2.
\end{equation}
Indeed, recalling the assumption {\bf (H1)} we see that for $h$ sufficiently small,
$B_{x,h}$ are star-shaped with respect to $x$ and that both the Lipschitz 
constants of their boundaries and the ratios of their diameters have common bounds.

Our goal is to replace $A_{x,h}$ by a matrix $R(x)$ which depends smoothly on $x$.
This will allow us to replace $A_{x,h}$ by $R(\pi z)$ in (\ref{1.uno}).
The desired estimate on $S^h$ then follows by summing
over a suitable family of cylinders.
The smoothness of $R$ will also play essential role in the key estimate
in Lemma \ref{lem3}.

\medskip

{\bf 2.} 
To define $R(x)$ consider a cut-off function 
$\vartheta\in \mathcal{C}_c^\infty ([0,1))$, with $\vartheta\geq 0$, 
$\vartheta$ constant in 
a neighbourhood of $0$, and $\int_0^1\vartheta = 1$.
For each $x\in S$ define:
$$\displaystyle\eta_x(z) = \frac{\vartheta(|\pi z - x|/h)}
{\int_{S^h}\vartheta(|\pi z - x|/h) ~\mbox{d}z}.$$
Then $\eta_x(z) = 0$ for $z\not\in B_{x,h}$ and:
$$\int_{S^h} \eta_x(z)~\mbox{d}z = 1,\qquad |\eta_x|\leq \frac{C}{h^n},
\qquad |\nabla_{x}\eta_x|\leq \frac{C}{h^{n+1}}$$
Define $R(x)$ as the average:
$$R(x) = \int_{S^h} \eta_x(z) \mbox{ skew}(\nabla u(z))~\mbox{d}z,$$
where $\mbox{skew} (F) = (F-F^T)/2$ denotes the skew-symmetric part of a given 
matrix $F$. Since $\int\eta_x = 1$, we have:
$$R(x) - A_{x,h} = \int_{S^h} \eta_x(z) 
\mbox{ skew}(\nabla u(z)- A_{x,h})~\mbox{d}z,$$
and by the Cauchy-Schwarz inequality, noting that $|\mbox{skew} (F)|\leq C |F|$
we obtain:
\begin{equation}\label{1.due}
|R(x) - A_{x,h}|^2 \leq 
C\left(\int_{S^h} \eta_x(z) |\nabla u(z)- A_{x,h}|~\mbox{d}z\right)^2
\leq \frac{C}{h^n}\int_{B_{x,h}} |D(u)|^2.
\end{equation}
To estimate the derivative of $R$ we use that:
$$\int_{S^h} \nabla_x\eta_x(z)~\mbox{d}z 
= \nabla_x\left(\int_{S^h} \eta_x(z)~\mbox{d}z\right)=0.$$
Thus:
$$\nabla R(x) = \int_{S^h} \left(\nabla_x\eta_x\right) \mbox{ skew}(\nabla u)
= \int_{S^h} \left(\nabla_x\eta_x\right) \mbox{ skew}(\nabla u- A_{x,h})$$
and by (\ref{1.uno}):
\begin{equation}\label{1.tre}
|\nabla R(x)|^2 \leq \int_{B_{x,h}}\left|\nabla_x\eta_x\right|^2\cdot
\int_{B_{x,h}}\left|\nabla u - A_{x,h}\right|^2 \leq \frac{C}{h^{n+2}}
\int_{B_{x,h}} |D(u)|^2.
\end{equation}
Similarly, we get for all $x'\in D_{x,h}$:
\begin{equation}\label{1.quattro}
|\nabla R(x')|^2 \leq \frac{C}{h^{n+2}} \int_{B_{x',h}} |D(u)|^2 \leq
\frac{C}{h^{n+2}}\int_{2B_{x,h}} |D(u)|^2,
\end{equation}
where $2B_{x,h} = \pi^{-1}(D_{x,2h})\cap S^h$. From this, by the fundamental theorem
of calculus:
$$|R(x'') - R(x)|^2 \leq \frac{C}{h^n}\int_{2B_{x,h}} |D(u)|^2 
\qquad \forall x''\in D_{x,h}.$$
In combination with (\ref{1.uno}) and (\ref{1.due}) this yields:
\begin{equation}\label{1.cinque}
\int_{B_{x,h}} |\nabla u(z) - R(\pi z)|^2~\mbox{d}z 
\leq C \int_{2B_{x,h}} |D(u)|^2.
\end{equation}
Now cover $S^h$ with a family $\{B_{x_i,h}\}_{i=1}^{N(h)}$ 
so that the covering number of $\{2B_{x_i,h}\}_{i=1}^{N(h)}$ is independent of $h$. 
A possible argument for the existence of such a covering goes as follows.
The surface $S$ is contained in the finite union of balls 
$\cup_{i=1}^{N(h)} B(x_i, h/2)$
where $k_i\in (\frac{h}{2}\mathbb{Z})^n$. Fix a one-to-one map $k_i\mapsto
x_i\in S\cap B(k_i, h/2)$, so that $S^h= \cup_i B_{x_i, h}$.
Then, if $z\in 2B_{x_i, h}$ there must be $\pi(z)\in B(x_i, 2h)$, so that
$|k_i - \pi(z)|\leq |k_i - x_i| + |\pi(z) - x_i| \leq 5h/2$. Therefore 
$k_i\in B(x,5h/2) \cap (\frac{h}{2}\mathbb{Z})^n$. The cardinality of this
last set is bounded by $10^n$, which must as well be a covering number for
the family $\{2B_{x_i,h}\}_{i=1}^{N(h)}$.

Summing (\ref{1.cinque}) 
over $i=1\ldots N$ proves (i).
Finally, integrating (\ref{1.quattro}) on $D_{x,h}$ we get;
$$\int_{D_{x,h}}  |\nabla R(x')|^2~\mbox{d}x' 
\leq \frac{C}{h^3}\int_{2B_{x,h}} |D(u)|^2,$$
and using the same covering as before we obtain (ii).
\end{proof}

Following the same argument, we will prove a uniform Poincar\'e inequality 
in thin domains - see Theorem \ref{th_uniform_poincare} in Appendix D.

\section{The key estimates}
\label{key_estimates}

Let $\bar{u}:S\longrightarrow \mathbf{R}^n$ be the average of $u$ in the 
normal direction:
\begin{equation}\label{average}
\bar{u}(x) = \fint_{-g_1^h(x)}^{g_2^h(x)} u(x+ t\vec n(x))~\mbox{d}t \qquad 
\forall x\in S.
\end{equation}
In this section we will establish four useful estimates on various components 
of $\bar{u}$ and their derivatives. 

The first estimate on $\nabla \bar{u}$, is an extension of the previous Theorem \ref{approx}:

\begin{lemma}\label{lem2}
Assume {\bf (H1)}. For every $u\in W^{1,2}(S^h,\mathbf{R}^n)$ there holds:
$$\|\nabla \bar{u} - R_{tan}\|_{L^2(S)} \leq
Ch^{1/2} \|u\|_{W^{1,2}(S^h)} + Ch^{-1/2}  \|D(u)\|_{L^2(S^h)},$$
where the subscript '$tan$' refers to the tangential components 
of the appropriate matrix valued function, that is:
$R_{tan}(x)\vec n(x) = 0$ and $R_{tan}(x)\tau = R(x)\tau$ for all 
$x\in S$ and $\tau\in T_xS$.
\end{lemma}
\begin{proof}
Through a direct calculation one checks that for every $x\in S$ and $\tau\in T_x S$ there holds:
\begin{equation*}
\begin{split}
&\left| \partial_\tau \bar{u}(x) - \fint_{-g_1^h(x)}^{g_2^h(x)} \nabla u(x + t\vec n(x))\cdot
\left\{\tau + t\partial_\tau \vec n(x)\right\}~\mbox{d}t\right| \\
&\qquad \leq \frac{C}{h} \left(|\partial_\tau g_1^h(x)| + |\partial_\tau g_2^h(x)|\right)\cdot
\int_{-g_1^h(x)}^{g_2^h(x)}|\partial_{\vec n} u(x + t\vec n(x))|~\mbox{d}t
\leq C\int_{-g_1^h(x)}^{g_2^h(x)} |\nabla u|~\mbox{d}t
\end{split}
\end{equation*}
and:
\begin{equation*}
\begin{split}
&\fint_{-g_1^h(x)}^{g_2^h(x)} \left|\nabla u(x + t\vec n(x))\cdot
\left(\tau + t\partial_\tau \vec n(x)\right) - R(x)\tau\right|~\mbox{d}t \\
&\qquad\qquad
 \leq C \int_{-g_1^h(x)}^{g_2^h(x)}|\nabla u|~\mbox{d}t + 
\fint_{-g_1^h(x)}^{g_2^h(x)} |\nabla u(x + t\vec n(x)) - R(x)|~\mbox{d}t.
\end{split}
\end{equation*}
Hence, by Theorem \ref{approx} (i):
\begin{equation*}
\begin{split}
\|\nabla \bar{u} - R_{tan}\|_{L^2(S)}^2 & \leq 
C\int_S\left\{h \int_{-g_1^h(x)}^{g_2^h(x)}|\nabla u|^2~\mbox{d}t +
h^{-1}\fint_{-g_1^h(x)}^{g_2^h(x)} |\nabla u - R\pi|^2~\mbox{d}t\right\}~\mbox{d}x\\
& \leq Ch \|\nabla u\|_{L^2(S^h)}^2 + Ch^{-1} \|D(u)\|_{L^2(S^h)}^2.
\end{split}
\end{equation*}
\end{proof}

In order to estimate the normal part $\bar{u}$, we will use the following bounds:

\begin{lemma}\label{trivial}
Recall that $\partial S^h = \partial^- S^h \cup \partial^+ S^h$, with:
\begin{equation}\label{bd_12}
\begin{split}
\partial^- S^h &= \{ x - g_1^h(x)\vec n(x) ; ~ x\in S\},\\
\partial^+ S^h &= \{ x + g_2^h(x)\vec n(x) ; ~ x\in S\}. 
\end{split}
\end{equation}
\begin{itemize}
\item[(i)] If {\bf (H1)} holds then $|\vec n^h(z) - \vec n(\pi(z))|\leq Ch$ for all 
$z\in\partial^+ S^h$ and $|\vec n^h(z) + \vec n(\pi(z))|\leq Ch$ for all 
$z\in\partial^- S^h$.
\item[(ii)] If {\bf (H2)} holds then:
\begin{equation*}
\begin{split}
|\vec n^h(z) + \vec n(\pi(z)) + \nabla g_1^h(\pi(z))|&\leq Ch^2 \qquad \forall z\in \partial^- S^h,\\
|\vec n^h(z) - \vec n(\pi(z)) + \nabla g_2^h(\pi(z))|&\leq Ch^2 \qquad \forall z\in \partial^+ S^h.
\end{split}
\end{equation*}
\end{itemize}
Let now $u\in W^{1,2}(S^h, \mathbf{R}^n)$. 
\begin{itemize}
\item[(iii)] ${\displaystyle |\partial_{\vec n} (u\cdot\vec n)(z)| \leq |D(u)(z)|}$ for all $z\in S^h$.
\item[(iv)] If {\bf (H1)} holds and $u\cdot\vec n^h=0$ on $\partial^+ S^h$,
then:
 $$\displaystyle \|u\cdot \vec n\|_{L^2(\partial^+ S^h)}\leq
Ch^{1/2}\|u\|_{W^{1,2}(S^h)}.$$
\item[(v)] If {\bf (H2)} holds and $u\cdot\vec n^h=0$ on $\partial S^h$:
\begin{equation*}
\begin{split}
\int_S |u(x - g_1^h(x)&\vec n(x))\cdot\nabla g_1^h(x) +
 u(x + g_2^h(x)\vec n(x))\cdot\nabla g_2^h(x)|^2 ~\mathrm{d}x\\
& \leq Ch \int_{S^h}|D(u)|^2 + Ch^3 \|u\|_{W^{1,2}(S^h)}^2.
\end{split}
\end{equation*}
\end{itemize}
\end{lemma}
\begin{proof}
(i) is obvious. To prove (ii) observe, for example, that
on $\partial^+ S^h$ the normal $\vec n^h(z)$ is parallel to
$\vec n(\pi(z)) - \nabla g_2^h(\pi(z)) + w$, where 
$|w|\leq C |g_2^h(\pi(x)) \nabla g_2^h(\pi(z))|
\leq Ch^2$. Normalising this vector  we conclude the second inequality in (ii). 
The first one follows in the same manner.

(iii) follows from: $\partial_{\vec n}(u\cdot\vec n) = D(u) \vec n \cdot \vec n$.

To prove (iv), use (i) and the trace theorem in Appendix D:
$$\|u\cdot\vec{n}\|_{L^2(\partial^+ S^h)} 
= \|u\cdot(\vec n - \vec n^h)\|_{L^2(\partial^+ S^h)}
\leq C h^{1/2}\|u\|_{W^{1,2}(S^h)}.$$

For (v), use (ii), (iii) and Theorem \ref{th_uniform_trace}:
\begin{equation*} 
\begin{split}
&\int_S |u(x + g_2^h(x)\vec n(x))\cdot\nabla g_2^h(x) +
 u(x - g_1^h(x)\vec n(x))\cdot\nabla g_1^h(x)|^2 ~\mathrm{d}x\\
& \leq \int_S |u(x + g_2^h(x)\vec n(x))\cdot\vec n(x) -
 u(x - g_1^h(x)\vec n(x))\cdot\vec n(x)|^2 ~\mathrm{d}x
+ Ch^4 \int_{\partial S^h} |u|^2 \\
& \qquad\qquad 
= \int_S\left|\int_{-g_1^h(x)}^{g_2^h(x)} \partial_{\vec n} 
(u\cdot \vec n) (x+t\vec n(x))
~\mathrm{d}t\right|^2 ~\mathrm{d}x + Ch^4 \int_{\partial S^h} |u|^2 \\
& \qquad \qquad \leq Ch \int_{S^h}|D(u)|^2 + Ch^3 \|u\|_{W^{1,2}(S^h)}^2.
\end{split}
\end{equation*}
\end{proof}


\begin{lemma}\label{lem1}
Assume {\bf (H1)} and let $u\in W^{1,2}(S^h,\mathbf{R}^n)$ satisfy $u\cdot\vec n^h = 0$
on $\partial^+ S^h$. Then:
$$\|\bar{u}\cdot \vec n\|_{L^2(S)}\leq Ch^{1/2} \|u\|_{W^{1,2}(S^h)}.$$
\end{lemma}
\begin{proof}
By Lemma \ref{trivial} (iv) and (i), for every $z=x+t\vec n(x)\in S^h$ we obtain:
\begin{equation*}
\begin{split}
|u(x+&t\vec n(x))\cdot\vec n(x)|^2  \leq 
\left(|u(x + g_2^h(x)\vec n(x))\cdot \vec n(x)|
+ \int_{-g_1^h(x)}^{g_2^h(x)}|D(u)|\right)^2\\
& \leq C \cdot \left|u(x +  g_2^h(x)\vec n(x))\cdot \left(\vec n(x) 
- \vec n^h(x+g_2^h(x)\vec n(x))\right)\right|^2
+ Ch \int_{-g_1^h(x)}^{g_2^h(x)}|D(u)|^2\\
& \leq Ch^2  |u(x + g_2^h(x)\vec n(x))|^2 + Ch  \int_{-g_1^h(x)}^{g_2^h(x)}|D(u)|^2.
\end{split}
\end{equation*}
Hence by Theorem \ref{th_uniform_trace}:
\begin{equation}\label{help}
\begin{split}
\|\bar{u}\cdot\vec n\|_{L^2(S)}^2  & \leq \frac{C}{h} 
\int_S\int_{-g_1^h(x)}^{g_2^h(x)} |u(x + t\vec n(x))\cdot\vec n(x)|^2~\mbox{d}t~\mbox{d}x\\
& \leq \frac{C}{h} \left(h^3 \|u\|_{L^2(\partial S^h)}^2 + h^2 \|D(u)\|_{L^2(S^h)}^2\right) 
\leq C h  \|\nabla u\|_{L^2(S^h)}^2.
\end{split}
\end{equation}
\end{proof}

The next, key estimate, is on the gradient of $\bar{u}\cdot\vec n$. It is obtained
using the divergence theorem on the surface $S$:

\begin{lemma}\label{lem3}
Assume {\bf (H1)} and let $u\in W^{1,2}(S^h,\mathbf{R}^n)$ satisfy $u\cdot\vec n^h = 0$
on $\partial^+ S^h$. Then:
\begin{equation*}
\begin{split}
\|\nabla (\bar{u}\cdot\vec n)\|_{L^2(S)}  + \|R\vec n\|_{L^2(S)} & \leq
C\left( \|\bar{u}\|_{L^2(S)} + \|u\|_{W^{1,2}(S^h)}  
+ h^{-1/2} \|D(u)\|_{L^2(S^h)}\right) \\
& \quad + C\left(h^{-1} \|u\|_{W^{1,2}(S^h)}\cdot\|D(u)\|_{L^2(S^h)}\right)^{1/2}.
\end{split}
\end{equation*}
\end{lemma}
\begin{proof}
First note that $\|R\vec n\|_{L^2(S)} = \|\vec n^T R_{tan}\|_{L^2(S)}$,
since $\vec n R\vec n=0$ and $R\in so(n)$.
To prove the desired estimate we use the Hilbert space identity:
$$\|a\|^2 + \|b\|^2 = \|a - b\|^2 + 2\langle a,b\rangle$$
with $a= \nabla (\bar{u}\cdot\vec n)$ and $ b=\vec n^T R_{tan}$.

Integration by parts shows that:
\begin{equation}\label{est3}
\begin{split}
\langle a,b\rangle & =\left|\int_S (\vec n R_{tan})\cdot\nabla(\bar{u}\cdot\vec n)\right|
\leq C \|\bar{u}\cdot\vec n\|_{L^2(S)}
\left(\|R\|_{L^2(S)} + \|\nabla(\vec n R_{tan})\|_{L^2(S)}\right) \\
& \leq C \|\bar{u}\cdot\vec n\|_{L^2(S)}\|R\|_{W^{1,2}(S)}\\
& \leq C \|\bar{u}\cdot\vec n\|_{L^2(S)}
\left(h^{-3/2} \|D(u)\|_{L^2(S^h)} + h^{-1/2} \|\nabla u\|_{L^2(S^h)}
\right)\\
& \leq C h^{-1} \|u\|_{W^{1,2}(S^h)} \cdot\|D(u)\|_{L^2(S^h)}
+ C \|u\|_{W^{1,2}(S^h)}^2,
\end{split}
\end{equation}
where we applied the divergence theorem, Theorem \ref{approx} and Lemma \ref{lem1}.

On the other hand $a=\vec n^T\nabla\bar{u} + \bar{u}\cdot \nabla\vec n$, so by 
Lemma \ref{lem2}:
\begin{equation}\label{est4}
\begin{split}
\|a-b\| & \leq C \left(\|\bar{u}\|_{L^2(S)} 
+ \|\nabla \bar{u} - R_{tan}\|_{L^2(S)}\right)\\ 
& \leq C\|\bar{u}\|_{L^2(S)} + C h^{1/2} \|u\|_{W^{1,2}(S^h)}
+ C h^{-1/2} \|D(u)\|_{L^2(S^h)}.
\end{split}
\end{equation}
Combining (\ref{est3}) and (\ref{est4}) proves the result.
\end{proof}

Finally, in presence of the stronger condition {\bf (H2)}, 
we have an additional bound:

\begin{lemma}\label{ass_h2}
Assume {\bf (H2)} and let $u\in W^{1,2}(S^h,\mathbf{R}^n)$, $u\cdot \vec n^h = 0$
on $\partial S^h$. Then:
$$\frac{1}{h}\int_S |\bar{u}\cdot \nabla (g_1^h + g_2^h)| 
\leq C h^{1/2}\|u\|_{W^{1,2}(S^h)} + C h^{-1/2} \|D(u)\|_{L^2(S^h)}.$$
\end{lemma}
\begin{proof}
We have:
\begin{equation*}
\begin{split}
\frac{1}{h} \int_S |\bar{u} &\cdot \nabla (g_1^h + g_2^h)|\\
&\leq \frac{1}{h} \int_S |u(x-g_1^h(x)\vec n(x))\cdot\nabla g_1^h(x) +
u(x+g_2^h(x)\vec n(x))\cdot\nabla g_2^h(x)| ~\mbox{d}x\\
&\quad + C\|u - \bar{u}\pi\|_{L^1(\partial S^h)}\\
&\leq Ch^{-1/2} \|D(u)\|_{L^2(S)} + C h^{1/2} \|u\|_{W^{1,2}(S^h)}
+ C h^{1/2}\|\nabla u\|_{L^{2}(S^h)}.
\end{split}
\end{equation*}
The last inequality follows from Lemma \ref{trivial} (v)
and from an easy bound: $\|u - \bar{u}\pi\|_{L^1(\partial S^h)}\leq
C h^{1/2}\|\nabla u\|_{L^{2}(S^h)}$.
\end{proof}

\section{A proof of main theorems}
\label{section_endproof}

In this section we will prove the uniform Korn's estimate:
\begin{equation}\label{korn_proof}
 \|u\|_{W^{1,2}(S^h)}\leq C \|D(u)\|_{L^2(S^h)},
\end{equation}
under the angle constraints (\ref{f1}) or (\ref{f2}).
We argue by contradiction;
assume thus that (\ref{korn_proof}) is not valid, for any uniform constant $C$.
Hence, there exist sequences $h_n\longrightarrow 0$ and 
$u^{h_n}\in W^{1,2}(S^{h_n})$ (for simplicity we will write $h$ 
instead of $h_n$) such that the assumptions of
Theorem \ref{th1} or \ref{th2} are satisfied, but:
\begin{equation}\label{contra}
h^{-1/2} \|u^h\|_{W^{1,2}(S^h)} = 1 \quad \mbox{ and } \quad 
h^{-1/2} \|D(u^h)\|_{L^2(S^h)} \longrightarrow 0 \qquad 
\mbox{ as } h\longrightarrow 0.
\end{equation}
For the proof of Theorem \ref{th1} we will assume that $u^h\cdot \vec n^h=0$ 
on $\partial^+ S^h$. The case of the tangency condition on $\partial^- S^h$
is proved exactly the same.

Notice that (\ref{contra}) immediately gives, through Lemmas \ref{lem1},
\ref{lem2} and \ref{lem3}, that:
\begin{eqnarray}\label{SM1}
& \displaystyle\lim_{h\to 0} \left(\|\bar{u}^h\cdot \vec n\|_{L^2(S)}
+ \|\nabla \bar{u}^h - R_{tan}^h\|_{L^2(S)}\right) =0,\label{sm1}\\
& \displaystyle\lim_{h\to 0} \left(\|\nabla (\bar{u}^h\cdot\vec n)\|_{L^2(S)}
+ \|R^h\vec n\|_{L^2(S)}\right) \leq C \limsup_{h\to 0} \|\bar{u}^h\|_{L^2(S)}.
\label{SM2}
\end{eqnarray}
Also, Lemma \ref{ass_h2} implies that under the assumption {\bf (H2)}:
\begin{equation}\label{SM3}
\lim_{h\to 0} \int_S |\bar{u}^h\cdot\nabla (g_1 + g_2)|=0,
\end{equation}
where we used that the sequence $\bar{u}^h$ is bounded in $L^1(S)$,
again in view of (\ref{contra}).

A contradiction will be derived in several steps.
In particular, the tangential component of the average $\bar{u}$:
$$\bar{u}^h_{tan}(x) = \bar{u}^h(x) - (\bar{u}^h\cdot \vec n) 
\cdot \vec n(x) \in T_x S.$$
will be estimated using the Korn inequality on hypersurfaces
(see Appendix C). The conditions (\ref{f1})
and (\ref{f2}) assumed in Theorems 
\ref{th1} and \ref{th2} will be used in full (not just for rotations 
as in Theorem \ref{th_very_weak}).

\bigskip

\noindent {\bf Proof of Theorems \ref{th1} and \ref{th2}.}

\noindent {\bf 1.} Applying Theorem \ref{th_surface_korn} to each tangent vector
field $\bar{u}^h_{tan}$, we obtain a sequence $v_0^h\in\mathcal{I}(S)$
such that:
\begin{equation*}
\|\bar{u}^h_{tan} - v_0^h\|_{W^{1,2}(S)} 
\leq C\|D(\bar{u}^h_{tan})\|_{L^2(S)}.
\end{equation*}
For every $x\in S$ and $\tau\in T_x S$ there holds:
\begin{equation*}
\begin{split}
|\partial_\tau\bar{u}^h_{tan}(x) \cdot\tau| &= |\partial_\tau\bar{u}^h(x) \cdot\tau
- (\bar{u}^h\cdot\vec n)(x)\cdot \partial_\tau\vec n(x)|\\
&\leq |\partial_\tau\bar{u}^h(x) - R^h(x)\tau| + C |(\bar{u}^h\cdot\vec n)(x)|,
\end{split}
\end{equation*}
as $R^h(x) \in so(n)$.
Thus, by (\ref{SM1}):
\begin{equation*}
\|D(\bar{u}^h_{tan})\|_{L^2(S)} \leq C\left(\|\nabla\bar{u}^h 
- R^h_{tan}\|_{L^2(S)} + \|\bar{u}^h\cdot\vec n\|_{L^2(S)}\right)
\longrightarrow 0 \qquad \mbox{ as } h\longrightarrow 0.
\end{equation*}
Therefore:
\begin{equation}\label{due.5}
\lim_{h\to 0}\|\bar{u}^h_{tan} - v_0^h\|_{W^{1,2}(S)}=0. 
\end{equation}

\medskip

{\bf 2.} Let $\mathbb P$ be the orthogonal projection (with respect to 
the $L^2(S)$ norm) of the space $\mathcal{I}(S)$
onto its subspace $V$, which we take to be the whole  $\mathcal{I}(S)$
in case of Theorem \ref{th1} and  $\mathcal{I}_{g_1, g_2}(S)$
in case of Theorem \ref{th2}.
Call $v_1^h = \mathbb{P}v_0^h \in V$ and $v_2^h = v_0^h - v_1^h\in V^\perp$.
In both cases (\ref{dist_angle}) implies:
\begin{equation}\label{tre}
\|u^h\|_{L^2(S^h)}\leq C \|u^h - v_1^h\pi\|_{L^2(S^h)}.
\end{equation}
 
We now prove that:
\begin{equation}\label{nove}
\lim_{h\to 0}  \|v_2^h\|_{L^2(S)} = 0.
\end{equation}
In case of Theorem \ref{th1}, when $V^\perp = \{0\}$, (\ref{nove}) is trivial,
so we concentrate on the case of Theorem \ref{th2}. Notice that then,
(\ref{SM3}) and (\ref{due.5}) yield:
\begin{equation}\label{h2_case}
\begin{split}
& \int_S |v_2^h\cdot\nabla(g_1 + g_2)|  = 
\int_S |v_0^h\cdot\nabla(g_1 + g_2)| \\ 
& \qquad \leq
C \|\bar{u}^h_{tan} - v_0^h\|_{L^1(S)} + C \int_S |\bar{u}^h\cdot\nabla(g_1 + g_2)|
\longrightarrow 0 \qquad \mbox{ as } h\longrightarrow 0.
\end{split}
\end{equation}
Since all norms in the finitely dimensional space $V^\perp$ are equivalent, we have:
\begin{equation}\label{dieci}
\|v_2^h\|_{L^{2}(S)} \leq C \int_S |v_2^h\cdot\nabla(g_1 + g_2)|.
\end{equation}
Indeed, the right hand side of (\ref{dieci}) provides a norm on the space 
in question.
Now, (\ref{h2_case}) and (\ref{dieci}) clearly imply (\ref{nove}).

\medskip

{\bf 3.} Using the Poincar\'e inequality 
on each segment $[-g_1^h(x), g_2^h(x)]$, and by (\ref{contra}):
\begin{equation}\label{7_help}
h^{-1/2}\|\bar{u}^h\pi - u^h\|_{L^2(S^h)}
\leq C h^{1/2} \|\nabla u^h\|_{L^2(S^h)}
\longrightarrow 0 \qquad \mbox{ as } h\longrightarrow 0.
\end{equation}
We now obtain convergence to $0$ of various quantities:
\begin{eqnarray}
&& h^{-1/2}\|\bar{u}_{tan}^h\pi - u^h\|_{L^2(S^h)} 
\leq h^{-1/2} \|\bar{u}^h\pi - u^h\|_{L^2(S^h)} +
C  \|\bar{u}^h\cdot\vec n\|_{L^2(S)} \longrightarrow 0 \nonumber\\
&& \qquad\qquad\qquad\qquad\qquad \qquad\qquad\quad
\mbox{ by (\ref{7_help}) and (\ref{SM1})},\nonumber\\
&&  h^{-1/2}\|v_0^h\pi - v_1^h\pi\|_{L^2(S^h)} \longrightarrow 0 
\qquad \mbox{ by  (\ref{nove})},\nonumber\\
&&  h^{-1/2}\|u^h - v_1^h\pi\|_{L^2(S^h)} \longrightarrow 0 
\qquad \mbox{ by  (\ref{due.5}) and convergences above}.\nonumber
\end{eqnarray}
Consequently, by (\ref{tre}):
\begin{equation}\label{ml2}
\lim_{h\to 0} h^{-1/2}\|u^h\|_{L^2(S^h)} = 0.
\end{equation}
Hence:
\begin{eqnarray}
&& \|\bar{u}^h\|_{L^2(S)} \longrightarrow 0, \label{ml4}\\
&& \|\nabla(\bar{u}^h\cdot\vec n)\|_{L^2(S)} + \|R^h\vec n\|_{L^2(S)}
 \longrightarrow 0 \qquad \mbox{ by  (\ref{SM2}) and (\ref{ml4})}, \label{ml5}\\
&&  \|v_0^h\|_{L^2(S)} \longrightarrow 0 
\qquad \mbox{ by  (\ref{due.5}) and (\ref{ml4})}. \label{ml6}
\end{eqnarray}
Because of the equivalence of all norms on the finitely dimensional space
$\mathcal{I}(S)$, (\ref{ml6}) implies:
\begin{equation}\label{ml7}
\lim_{h\to 0} \|v_0^h\|_{W^{1,2}(S)} = 0.
\end{equation}
Now, we may estimate the quantity $h^{-1/2}\|\nabla u^h\|_{L^2(S^h)}$
by the following norms: $h^{-1/2}\|\nabla u^h - R^h\pi\|_{L^2(S^h)}$,
$\|R^h\vec n\|_{L^2(S)}$, $\|R^h_{tan} - \nabla \bar{u}^h\|_{L^2(S)}$,
$\|\nabla(\bar{u}^h\cdot\vec n)\|_{L^2(S)}$,
$\|\nabla \bar{u}^h_{tan} - \nabla v_0^h\|_{L^2(S)}$,
$\|\nabla v_0^h\|_{L^2(S)}$, and use 
Theorem \ref{approx}, (\ref{ml5}),  (\ref{SM1}),  (\ref{due.5}) and  (\ref{ml7})
to conclude that:
\begin{equation*}
\lim_{h\to 0} h^{-1/2}\|\nabla u^h\|_{L^2(S^h)} = 0.
\end{equation*}
Together with (\ref{ml2}) this contradicts (\ref{contra}).
\endproof

\section{Estimates without Killing fields}
\label{section_weak_bound}

In this section we prove Theorem \ref{th_very_weak}.
The first step is to give a bound for the distance of $u$ from the generators 
of rigid motions in $\mathbf{R}^n$.
This follows from Theorem \ref{approx}
and the uniform Poincar\'e inequality in Theorem \ref{th_uniform_poincare}:

\begin{lemma}\label{cor4.2}
Assume {\bf (H1)}. For every $u\in W^{1,2}(S^h,\mathbf{R}^n)$ there exists a linear function 
$v(z) = Az + b$, $A\in so(n)$, $b\in\mathbf{R}^n$, such that:
$$\|u-v\|_{W^{1,2}(S^h)} \leq Ch^{-1}\|D(u)\|_{L^2(S^h)}.$$
\end{lemma}
\begin{proof}
Recall the results of Theorem \ref{approx} and define:
$$A=\fint_S R(x) ~\mbox{d}x \in so(n).$$
By Theorem \ref{approx} and the Poincar\'e inequality on $S$, we obtain:
\begin{equation}\label{1.dieci}
\begin{split}
\int_{S^h} |\nabla u - A|^2 & \leq C\left\{\int_{S^h}|\nabla u - R\pi|^2 + 
h\int_S |R(x) - A|^2 ~\mbox{d}x \right\} \\ & \leq
C\left\{\int_{S^h}|D(u)|^2 + h  \int_S |\nabla R|^2 \right\}
\leq C h^{-2}\int_{S^h}|D(u)|^2.
\end{split}
\end{equation}
We now apply 
Theorem \ref{th_uniform_poincare} to the function $u(z) - Az$, by which
for some $b\in\mathbf{R}^n$ there holds:
\begin{equation}\label{1.undici}
\int_{S^h}|u(z) - A z -b|^2 ~\mbox{d}z \leq C \int_{S^h} |\nabla u - A|^2
\leq Ch^{-2}\int_{S^h} |D(u)|^2. 
\end{equation}
Now (\ref{1.dieci}) and (\ref{1.undici}) imply the result.
\end{proof}

\bigskip

\noindent {\bf Proof of Theorem \ref{th_very_weak}.}
The proof of part (i) will be carried out assuming that 
$u\cdot\vec n^h=0$ on $\partial^+ S^h$. 
For the other case ($u\cdot\vec n^h=0$ on $\partial^- S^h$) the argument
is the same.

\medskip
\noindent 
{\bf 1.} We argue by contradiction. If (\ref{very_weak}) was not true, then there would be 
sequences $h_n\longrightarrow 0$ and $u^{h_n}\in W^{1,2}(S^{h_n})$
satisfying the conditions in (i) or (ii) and such that:
\begin{equation}\label{dodici}
h^{-1/2} \|u^h\|_{W^{1,2}(S^h)} = 1,
\end{equation}
\begin{equation}\label{1.tredici}
h^{-3/2} \|D(u^h)\|_{L^2(S^h)} \longrightarrow 0 \qquad \mbox{ as } h\longrightarrow 0
\end{equation}
(to simplify the notation, we write $h$ instead of $h_n$). By Lemma \ref{cor4.2},
there exists a sequence $v^h(z)= A^h z + b^h$, $A^h\in so(n)$, $b^h\in\mathbf{R}^n$, such that:
\begin{equation}\label{quattordici}
h^{-1/2} \|u^h - v^h\|_{W^{1,2}(S^h)}\longrightarrow 0 \qquad 
\mbox{ as } h\longrightarrow 0.
\end{equation}
Because of (\ref{dodici}), the sequence $h^{-1/2} v^h$ is bounded in 
$W^{1,2}(S^h)$ and so, without loss of generality, we may assume that:
\begin{equation}\label{quindici}
A^h\longrightarrow A\in so(n), \qquad
b^h\longrightarrow b\in \mathbf{R}^n \qquad \mbox{ as } h\longrightarrow 0.
\end{equation}
Moreover, by (\ref{dodici}) and (\ref{quattordici}):
$$\lim_{h\to 0} h^{-1/2}\|v^h\|_{W^{1,2}(S^h)} 
= \lim_{h\to 0} h^{-1/2}\|u^h\|_{W^{1,2}(S^h)} = 1,$$
and therefore:
\begin{equation}\label{Ab_nonzero}
|A| + |b| \neq 0.
\end{equation}

\medskip 

{\bf 2.}
We now prove that if {\bf (H1)} holds together with $u^h\cdot\vec n^h=0$ on
$\partial^+ S^h$, then we must have $Ax+b\in\mathcal{R}(S)$. 
Indeed, by Theorem \ref{th_uniform_trace} and Lemma \ref{trivial} (iv):
\begin{equation*}
\begin{split}
&\|v^h\cdot\vec n\|_{L^2(\partial^+ S^h)} \leq \|u^h - v^h\|_{L^2(\partial^+ S^h)}
+ \|u^h\cdot\vec n\|_{L^2(\partial^+ S^h)}\\
&\qquad\qquad
\leq C\left(h^{-1/2}\|u^h - v^h\|_{W^{1,2}(S^h)} + h^{1/2}\|u^h\|_{W^{1,2}(S^h)}\right)
\longrightarrow 0 \quad \mbox{ as } h\longrightarrow 0,
\end{split}
\end{equation*}
where the convergence above follows from (\ref{dodici}) and (\ref{quattordici}). Thus:
\begin{equation*}
\int_S |(Ax+b)\cdot\vec n(x)|^2~\mbox{d}x = \lim_{h\to 0} \int_S |v^h(x)\cdot\vec n (x)|^2 ~\mbox{d}x
= \lim_{h\to 0}\|v^h\cdot\vec n\pi\|_{L^2(\partial^+ S^h)}^2 =0.
\end{equation*}

We now prove that if {\bf (H2)} holds, together with
$u^h\cdot\vec n^h=0$ on $\partial S^h$,
then $Ax+b\in\mathcal{R}_{g_1,g_2}(S)$.
By  Theorem \ref{th_uniform_trace} and Lemma \ref{trivial} (v):
\begin{equation}\label{sedici}
\begin{split}
&\frac{1}{h^2} \int_S|v^h(x+g_2^h(x)\vec n(x))\cdot \nabla g_2^h(x)
+ v^h(x - g_1^h(x)\vec n(x))\cdot \nabla g_1^h(x)|^2 ~\mathrm{d}x\\
&\leq \frac{1}{h^2}\Big\{ Ch^2\|v^h - u^h\|_{L^2(\partial S^h)}^2 \\
&\quad + \int_S|u^h(x+g_2^h(x)\vec n(x))\cdot \nabla g_2^h(x)
+ u^h(x - g_1^h(x)\vec n(x))\cdot \nabla g_1^h(x)|^2 ~\mathrm{d}x\Big\}\\
&\leq \frac{C}{h} \left\{\|v^h - u^h\|_{W^{1,2}(S^h)} + \|D(u^h)\|_{L^2(S^h)}^2 
 + h^2 \|u^h\|_{W^{1,2}(S^h)}^2\right\}
\longrightarrow 0 \qquad \mbox{ as } h\longrightarrow 0,
\end{split}
\end{equation}
where (\ref{quattordici}) with (\ref{1.tredici}) justify the convergence.
Hence, by (\ref{sedici}):
\begin{equation*}
\int_S |(Ax+b) \cdot \nabla(g_1 + g_2)(x)|^2 ~\mbox{d}x
= \lim_{h\to 0} \frac{1}{h^2} \int_S |v^h \cdot \nabla(g_1^h + g_2^h)|^2= 0 
\end{equation*}

\medskip

{\bf 3.} We see that in both cases (i) and (ii) there holds
(using condition (\ref{dist_angle})):
$$\|u^h\|_{L^2(S^h)}\leq C \|u^h - (A\pi (z) + b)\|_{L^2(S^h)}.$$
Thus, by (\ref{quattordici}) and (\ref{quindici}):
\begin{equation*}
\begin{split}
& h^{-1/2}\|u^h\|_{L^2(S^h)} \leq C  h^{-1/2} \|u^h - (A\pi (z) + b)\|_{L^{2}(S^h)} \\
& \qquad \leq C h^{-1/2} \|u^h - v^h\|_{L^{2}(S^h)} 
+ Ch^{-1/2} \|v^h - (A\pi (z) + b)\|_{L^{2}(S^h)}
\longrightarrow 0.
\end{split}
\end{equation*}
We deduce that $\lim_{h\to 0} h^{-1/2}\|v^h\|_{L^2(S^h)}=0$ as well,
which contradicts (\ref{Ab_nonzero}).
\endproof

\section{Appendix A - The Korn-Poincar\'e inequality in a fixed domain}
\label{easy_korn}

In this section $\Omega\subset \mathbf{R}^n$ is a fixed open, bounded domain with Lipschitz
boundary. For $x\in\partial \Omega$, by $\vec n_\Omega(x)$ we denote the outward unit normal
to $\partial\Omega$ at $x$.
We first recall the standard Korn inequality \cite{ciarbookvol1, GS, FJMgeo}:

\begin{theorem}\label{korn_standard}
(i) There holds:
$$\left\{u\in L^2(\Omega,\mathbf{R}^n);~~ D(u)\in L^2(\Omega, M^{n\times n})\right\}
= W^{1,2}(\Omega,\mathbf{R}^n),$$
and the following equivalence of norms:
$$\|u\|_{W^{1,2}(\Omega)} \leq C_\Omega 
\left(\|u\|_{L^2(\Omega)} + \|D(u)\|_{L^2(\Omega)}\right)
\leq C_\Omega^2  \|u\|_{W^{1,2}(\Omega)}.$$
(ii) For every $u\in W^{1,2}(\Omega,\mathbf{R}^n)$ there exists $A\in so(n)$
and $b\in \mathbf{R}^n$ so that:
$$\|u - (Ax+b)\|_{W^{1,2}(\Omega)}\leq  C_\Omega  \|D(u)\|_{L^2(\Omega)}.$$
The constants $C_\Omega$ above depend only on the domain $\Omega$ and not on $u$.
\end{theorem}

Notice that Theorem \ref{korn_standard} (ii) implies that for each $u\in W^{1,2}(\Omega,\mathbf{R}^n)$ 
satisfying the orthogonality condition:
$$\int_\Omega u\cdot v = 0 \qquad \forall v\in\mathcal{R}(\Omega) = 
\left\{Ax+b;~~ A\in so(n), b\in\mathbf{R}^n\right\}$$
one has:
$$\|u\|_{W^{1,2}(\Omega)}\leq  C_\Omega  \|D(u)\|_{L^2(\Omega)}.$$
The same is true if we restrict our attention 
to vector fields tangential on $\partial\Omega$.
Define:
$$\mathcal{R}_\partial(\Omega) = 
\left\{v\in\mathcal{R}(\Omega);~~ v\cdot\vec n_\Omega=0 \mbox{ on } 
\partial\Omega\right\}.$$

\begin{theorem}\label{naive_korn}
For every $u\in W^{1,2}(\Omega,\mathbf{R}^n)$ such that $u\cdot\vec n_\Omega=0$ on 
$\partial\Omega$ and:
\begin{equation}\label{6.2.uno}
\int_\Omega u\cdot v = 0 \qquad \forall v\in\mathcal{R}_\partial (\Omega),
\end{equation}
there holds:
$$\|u\|_{W^{1,2}(\Omega)}\leq  C_\Omega \|D(u)\|_{L^2(\Omega)},$$
and the constant $C_\Omega$ depends only on $\Omega$.
\end{theorem}
\begin{proof}
We argue by contradiction, starting with a sequence $u_n\in W^{1,2}(\Omega)$
satisfying $u_n\cdot\vec n_\Omega=0$ on $\partial\Omega$, (\ref{6.2.uno}) and:
\begin{equation}\label{help2}
\|u_n\|_{W^{1,2}(\Omega)} = 1, \qquad \|D(u_n)\|_{L^2(\Omega)} \longrightarrow 0
\quad \mbox{ as } \quad n\longrightarrow\infty.
\end{equation}
Without loss of generality, $u_n$ converges hence weakly to some $u$ in $W^{1,2}(\Omega)$, and
the convergence is strong in $L^2(\Omega)$. Clearly $u\cdot\vec n_\Omega=0$ on $\partial\Omega$
and (\ref{6.2.uno}) still holds. By Theorem \ref{korn_standard} (ii), there exist sequences 
$A_n\in so(n)$ and $b_n\in \mathbf{R}^n$ so that $u_n - (A_n x + b_n)$ converges
to $0$ in $W^{1,2}(\Omega)$.

Therefore $A_nx+b_n$ converges weakly to $u$ in $W^{1,2}(\Omega)$ and we see that 
$u\in\mathcal{R}_\partial (\Omega)$.
By (\ref{6.2.uno}) there hence must be $u=0$ and $u_n$ converges then (strongly) to $0$ in
$W^{1,2}(\Omega)$. This contradicts the first condition in (\ref{help2}).
\end{proof}

\begin{example}
Let $\Omega=B_1\subset\mathbf{R}^3.$ Since $A\in so(3)$, there must be 
$Ax= a\times x$, for some $a\in\mathbf{R}^3$ and we obtain:
$$\mathcal{R}_\partial (B_1) = \{a\times x; ~ a\in\mathbf{R}^3.\}$$
Condition (\ref{6.2.uno}) reads:
$$0 = \int_{B_1} (a\times x)\cdot u(x) ~\mbox{d}x = a\cdot \int_{B_1} x\times u(x) ~\mbox{d}x
\qquad \forall a\in \mathbf{R}^3.$$
Thus the class of functions $u$ for which the hypotheses of Theorem \ref{naive_korn} are
satisfied is the following:
$$\left\{u\in W^{1,2}(\Omega); ~~ u\cdot\vec n_\Omega = 0 \mbox{ on } \partial B_1, ~
\int_{B_1} x\times u(x) ~\mbox{d}x = 0 \right\}.$$
\end{example}

\bigskip

As observed in the next result, condition (\ref{6.2.uno}) is not void if and only if our
bounded domain $\Omega$ is rotationally symmetric.

\begin{theorem}\label{lemma_rotations}
If $\mathcal{R}_\partial (\Omega)\neq \{0\}$ then $\Omega$ must be rotationally symmetric. 
\end{theorem}
\begin{proof}
Let $v(x) = Ax+b\in\mathcal{R}_\partial (\Omega).$ We will prove that the flow
generated by the tangent vector field $v_{|\partial\Omega}$ is a rotation.

Since $A\in so(n)$ we have that $\mathbf{R}^n = Ker(A) \oplus Im(A)$ is an orthogonal decomposition of 
$\mathbf{R}^n$. Write $b=b^{ker} + A b_0$, $b^{ker} \in Ker (A)$, 
and consider the translated domain $\Omega_0 = \Omega + b_0$. Now:
$$Ax + b = A(x+b_0) + b^{ker}\qquad \forall x\in\Omega,$$
so $y\mapsto Ay + b^{ker}$ is a tangent vector field on $\partial\Omega_0$.
Consider the flow $\alpha$ which this field generates in $\mathbf{R}^n$:
$$\left\{\begin{array}{l} \alpha'(t) = A\alpha(t) + b^{ker}\\
\alpha(0) \in \partial\Omega_0.
\end{array}\right.$$
Then $\alpha(t) = \beta(t) + \delta(t)$, where:
$$\left\{\begin{array}{lll} \beta'(t) = A\beta(t), & \beta(0)\in Im(A) & \\
\delta'(t) = b^{ker}, & \delta(0)\in Ker(A), & \beta(0) + \delta(0) = \alpha(0).
\end{array}\right.$$
Notice that:
$$\frac{d}{dt}|\beta(t)|^2 = 2 \beta(t)\cdot A\beta(t) = 0,$$
so $\beta(t)$ remains bounded, while $\delta(t) = \delta(0) + t b^{ker}$ is unbounded
for $b^{ker}\neq 0$. Since $\alpha(t)\in\partial\Omega_0$ for all $t\geq 0$, 
there must be $b^{ker}=0$. Hence the flow $\alpha$ is a rotation (generated by $A\in so(n)$)
on $\partial\Omega_0$, which proves the claim.
\end{proof}

From the proof above it follows that each $v\in\mathcal{R}_\partial(\Omega)$ has the form
$v(x)=A(x+b_0)$, $A\in so(n)$, $b_0\in\mathbf{R}^n$. We thus obtain the following characterisation
when $\Omega\subset\mathbf{R}^3$:
$$\mathcal{R}_\partial (\Omega) = 
\left\{\begin{array}{ll} \{0\} & \mbox{ if } ~ \Omega \mbox{ has no rotational symmetry}\\
\mbox{a 1-parameter family} & \mbox{ if } ~ \Omega \mbox{ has one rotational symmetry}\\
\mbox{a 3-parameter family} & \mbox{ if } ~ \Omega = B_r.
\end{array}\right.$$

\section{Appendix B - The uniform Korn inequality} 
\label{uniform_korn}

Throughout this section we will make the following assumptions on $\Omega$:
\begin{equation*} \mathbf{(\Omega H)}\left[
\begin{minipage}{14.5cm}
\begin{itemize}
\item[(i)] $\Omega$ is an open, bounded subset of $\mathbf{R}^n$, star-shaped with respect to 
the origin.
\item[(ii)] There exists $L>0$ such that the following holds.  
For every $x\in\Omega\setminus \{0\}$,
denote by $p(x)$ the unique point on $\partial\Omega$, with the property
that the segment $[0, p(x)]$ contains $x$.
Then:
$$|p(x) - x| \leq L\mbox{dist }(x,\partial\Omega).$$
\end{itemize}
\end{minipage}\right.
\end{equation*}
Our goal is to prove:
\begin{theorem}\label{app_main_korn}
For every $u\in W^{1,2}(\Omega,\mathbf{R}^n)$  there exists $A\in so(n)$ such that:
\begin{equation}\label{app_uno}
\|\nabla u - A\|_{L^2(\Omega)} \leq C_{n,L} \|D(u)\|_{L^2(\Omega)},
\end{equation}
and the constant $C_{n,L}$ depends only on $n$ and (in nondecreasing manner) 
on $L$.
\end{theorem}
Our proof is essentially a combination of the arguments in \cite{OK, kufner}, 
where we need to keep track of the magnitude of various constants, and of \cite{FJMgeo}.
In \cite{FJMgeo}, the $L^2$ distance of $\nabla u$ from a single proper rotation is
estimated in terms of the $L^2$ norm of the pointwise distance of $\nabla u$ from 
the space of proper rotations $SO(n)$. Note that $so(n)$ is the tangent space to $SO(n)$
at $\mbox{Id}$. Hence (\ref{app_main_korn}) can be seen as the ``linear'' version of
the result in \cite{FJMgeo}.

For convenience of the reader, we present the proof of Theorem \ref{app_uno}.
A similar line of proof was adopted in \cite{Grisorods} 
where the constant $C_{n,L}$ in (\ref{app_uno}) (or, more recently, its $L^p$ counterpart in
\cite{Grisonew}) has been calculated explicitly,
for domains which are star-shaped with respect to a ball. 

\begin{lemma}\label{lemLcond}
Let $\Omega$ be an open, bounded subset of $\mathbf{R}^n$.
\begin{itemize}
\item[(i)] If $B_r\subset \Omega\subset B_R$ and $\Omega$ is star-shaped with respect to $B_r$, 
then $\mathbf{(\Omega H)}$ holds with $L=R/r$.
\item[(ii)] Conversely, if $\Omega$ satisfies $\mathbf{(\Omega H)}$ then it is star-shaped with respect 
to a ball $B_r$ such that, calling $R=\min \{\tilde R; ~\Omega\subset B_{\tilde R}\}$, the ratio
$R/r$ depends only on $L$, in nondecreasing manner.
\end{itemize}
\end{lemma}
\begin{proof}
(i) is immediate.  To prove (ii), fix $L$ sufficiently large. 
For each $x\in\mathbf{R}^n\setminus \{0\}$ define the 'diamond'
$D_x$ obtained by rotating the right triangle with vertexes $0, a, x$ and angle $\angle a0x = \alpha$,
so that $|x|/|a|=L$, around its hypotenuse $[0,x]$.
The property $\mathbf{(\Omega H)} (ii)$ can then be translated to: $D_x\subset \Omega$ 
for every $x\in\Omega$.

Let $x_1\in \partial\Omega$ be such that $|x_1|=R$. Then $D_{x_1}\subset\Omega$.  Let $x_2=a$ 
from the construction of $D_{x_1}$.  Clearly $x_2\in \Omega$ and hence $D_{x_2}\subset \Omega$.
Proceed in this manner, constructing diamonds $\{D_{x_i}\}_{i=1}^N$, 
with equal angles at the origin and all $x_i$ in the same $2$d subspace of 
$\mathbb{R}^n$.  
After finitely many steps of this procedure we will have $x_N\in D_{x_1}$
and $B_{\tilde r}\subset \Omega$ with $\tilde r =|x_N| = R/L^{N-1}$, 
where $N= [2\pi/\alpha]$.
An easy argument now shows that $\Omega$ is star-shaped with respect to $B_r$ for any
$r\leq\tilde r/L$. Namely, taking $x\in\partial\Omega$, the convex hull of
$B_r\cup\{x\}$ is contained in  $D_x\cup B_{\tilde r} \subset \Omega$.
Therefore, one can take $r=R/L^{(2\pi/\alpha)}$, so:
$$\frac{R}{r} = L^{\frac{2\pi}{\arccos(1/L)}},$$ which is a non-decreasing function of $L$.
\end{proof}

\begin{lemma}\label{app_kufner}
For every $\phi\in W^{1,2}(\Omega)$ there holds:
$$\int_\Omega |\phi|^2 \leq C_{n,L} \left(\int_{B_r} |\phi|^2 + \int_\Omega |\nabla \phi|^2
\mathrm{dist}^2(x,\partial\Omega) ~\mathrm{d}x\right).$$
\end{lemma}
\begin{proof}
Without loss of generality we may assume that $\phi\in \mathcal{C}^\infty(\mathbf{R}^n)$. 
We adopt the proof of Theorem 8.2. in \cite{kufner}.
Let $R$ and $r$ be as in Lemma \ref{lemLcond}.

Let $\theta:[0,\infty)\longrightarrow [0,1]$ be a smooth non-decreasing function
satisfying:
\begin{equation*}
\begin{split}
&\theta(s) = 0 \quad \mbox{ for } s\leq \frac{r}{4}, \qquad 
\theta(s) = 1 \quad \mbox{ for } s\geq \frac{r}{2},\\
&|\theta'(s)|\leq \frac{8}{r} \quad \mbox{ for } s\geq 0. 
\end{split}
\end{equation*}
Fix a point $p\in\partial\Omega$ and consider the function $\theta \phi$ on the
segment $[0,p]$ joining the origin and $p$. Using Hardy's inequality \cite{kufner} and 
condition $\mathbf{(\Omega H)}$ we obtain:
\begin{equation*}
\begin{split}
\int_{r/2}^{|p|}& |\phi|^2 ~\mbox{d}|x| \leq \int_0^{|p|} |\theta\phi|^2 ~\mbox{d}|x| 
\leq 4 \int_0^{|p|} \left|\frac{\partial(\theta\phi)}
{\partial |x|}\right|^2\cdot \Big||p| - |x|\Big|^2 ~\mbox{d}|x| \\
&\leq 8L^2 \left(\int_{r/4}^{r/2} |\theta'|^2 |\phi|^2 
\mbox{dist}^2(x,\partial\Omega) + \int_{r/4}^{|p|} |\nabla\phi|^2 
\mbox{dist}^2(x,\partial\Omega)~\mbox{d}|x|\right)\\
&\leq C_{n, L, R/r} \left(\int_{r/4}^{r/2} |\phi|^2 
+ \int_{r/4}^{|p|} |\nabla\phi|^2 
\mbox{dist}^2(x,\partial\Omega)~\mbox{d}|x|\right).
\end{split}
\end{equation*}
Hence, also:
\begin{equation*}
\int_{r/2}^{|p|} |x|^{n-1} |\phi|^2 
\leq C_{n,L, R/r} \left(\int_{r/4}^{r/2}
 |x|^{n-1}|\phi|^2 + \int_{r/4}^{|p|}  |x|^{n-1}
|\nabla\phi|^2\mbox{dist}^2(x,\partial\Omega)\right),
\end{equation*}
which after integration in spherical coordinates gives:
\begin{equation}\label{app_kufner_1}
\begin{split}
\int_{\Omega\setminus B_{r/2}} |\phi|^2 ~\mbox{d}x 
\leq C_{n,L,R/r} \Big(&\int_{B_{r/2}\setminus B_{r/4}}|\phi|^2 \\
& + \int_{\Omega\setminus B_{r/4}}|\nabla \phi|^2\mbox{dist}^2(x,\partial\Omega)~\mbox{d}x\Big).
\end{split}
\end{equation}

Since $C_{n,L,R/r} = C_{n,L}$ in view of Lemma \ref{lemLcond}, the result follows by (\ref{app_kufner_1}).   
\end{proof}

\begin{theorem}\label{app_fjm}
For every $\phi\in W^{1,2}(\Omega)$ there exists $a\in\mathbf{R}$ such that:
$$\int_\Omega |\phi-a|^2 \leq C_{n,L} \int_\Omega |\nabla \phi|^2 
\mathrm{dist}^2(x,\partial\Omega) ~\mathrm{d}x.$$
\end{theorem}
\begin{proof}
We adopt the method of proof from Theorem 3.1. in \cite{FJMgeo}.
Again, let $R$ and $r$ be as in Lemma \ref{lemLcond}.
By the Poincar\'e inequality we obtain:
\begin{equation}\label{app_fjm_1}
\int_{B_{r/2}} \left|\phi - \fint_{B_{r/2}} \phi\right|^2 
\leq C_{n} r^2 \int_{B_{r/2}} |\nabla \phi|^2
\leq C_n \int_{B_{r/2}} |\nabla \phi|^2 \mbox{dist}^2(x,\partial\Omega)~\mbox{d}x.
\end{equation}
Applying Lemma \ref{app_kufner} to the function $\phi - \fint_{B_{r/2}} \phi$ on $\Omega$ 
we therefore get:
\begin{equation*}
\begin{split}
\int_{\Omega} \left|\phi - \fint_{B_{r/2}} \phi\right|^2 
& \leq C_{n,L} \left(\int_{B_{r/2}} \left|\phi - \fint_{B_{r/2}} \phi\right|^2 + 
\int_\Omega |\nabla \phi|^2\mbox{dist}^2(x,\partial\Omega)~\mbox{d}x\right)\\
&\leq C_{n,L} \int_{\Omega} |\nabla \phi|^2 \mbox{dist}^2(x,\partial\Omega)~\mbox{d}x,
\end{split}
\end{equation*}
where the last inequality follows from (\ref{app_fjm_1}).
\end{proof}

We now recall the following result from \cite{OK}. For convenience of 
the reader, we reproduce its short proof.
\begin{lemma}\label{app_oleinik}
Let $\phi\in W^{1,2}(\Omega)$ be such that $\Delta \phi = 0$ in 
$\mathcal{D}'(\Omega)$. Then:
$$\int_\Omega |\nabla \phi|^2 \mathrm{dist}^2(x,\partial\Omega)~\mathrm{d}x
\leq 4  \int_\Omega |\phi|^2.$$
\end{lemma}
\begin{proof}
Fix a small $\epsilon>0$ and integrate the equation $\Delta\phi = 0$ against
the scalar function $(\mbox{dist}(x,\partial\Omega) - \epsilon)^2 \phi$, 
over the set $\Omega_\epsilon = 
\{x\in\Omega; ~\mbox{dist}(x,\partial\Omega)>\epsilon\}$.
Integrating by parts we obtain:
\begin{equation*}
\begin{split}
\int_{\Omega_\epsilon} (\mbox{dist}(x,\partial\Omega) -\epsilon)^2 
|\nabla \phi|^2~\mbox{d}x &= 
- \int_{\Omega_\epsilon} 2\mbox{dist}(x,\partial\Omega) -\epsilon)\phi(x)
\nabla\mbox{dist}(\cdot,\partial\Omega)\cdot\nabla \phi ~\mbox{d}x \\
&\leq 2\int_{\Omega_\epsilon}|\phi|^2 + \frac{1}{2}\int_{\Omega_\epsilon}
 (\mbox{dist}(x,\partial\Omega) -\epsilon)^2 
|\nabla \phi|^2~\mbox{d}x.
\end{split}
\end{equation*}
where we have used the binomial formula and the fact that
$|\nabla\mbox{dist}(\cdot,\partial\Omega)|\leq 1$.
The above implies:
$$ \int_{\Omega_\epsilon} (\mbox{dist}(x,\partial\Omega) -\epsilon)^2 
|\nabla \phi|^2~\mbox{d}x \leq  4\int_{\Omega}|\phi|^2,$$
and proves the lemma upon passing $\epsilon\to 0$.
\end{proof}

\bigskip

\noindent {\bf Proof of Theorem \ref{app_main_korn}.}

\noindent The left hand side of (\ref{app_uno}) represents the distance in $L^2(\Omega)$
of $\nabla u$ from the closed subspace of constant functions $A\in so(n)$. Since the distance
function is continuous, we may without loss of generality assume that 
$u\in\mathcal{C}^\infty(\mathbf{R}^n,\mathbf{R}^n)$. 

{\bf 1.} Consider the problem:
\begin{equation*}
\left\{\begin{array}{ll} \Delta v = \Delta u & \mbox{ in } \Omega,\\ 
v=0 & \mbox{ on } \partial\Omega.\end{array}\right.
\end{equation*}
Since:
\begin{equation}\label{useful_formula}
\Delta u = 2 \mbox{ div} \left\{D(u) - \frac{1}{2}(\mbox{tr } D(u))\cdot  \mbox{Id}\right\},
\end{equation}
we see that:
$$\int_\Omega |\nabla v|^2 = 2 \int_\Omega\nabla v : 
\left(D(u) - \frac{1}{2}(\mbox{tr } D(u))\cdot  \mbox{Id}\right)
\leq 4 \|\nabla v\|_{L^2(\Omega)} \|D(u)\|_{L^2(\Omega)}.$$
Therefore:
\begin{equation}\label{app_proof_1}
\|\nabla v\|_{L^2(\Omega)}\leq 4 \|D(u)\|_{L^2(\Omega)}.
\end{equation}

{\bf 2.} The remaining part $w=u-v$ is harmonic: $\Delta w = 0$ in $\Omega$. 
Hence, the components of $D(w)$ are also harmonic, and Lemma \ref{app_oleinik}
implies:
\begin{equation}\label{app_cov}
\int_\Omega |\nabla D(w)|^2 \mbox{dist}^2(x,\partial\Omega)~\mbox{d}x 
\leq 4 \int_\Omega |D(w)|^2.
\end{equation}
Notice that the components of $\nabla^2 w$ are linear combinations 
of components of $\nabla D(w)$, namely:
$[\nabla^2 w^k]_{ls} = \frac{\partial}{\partial x_l} [D(w)]_{ks} + 
\frac{\partial}{\partial x_s} [D(w)]_{kl} - \frac{\partial}{\partial x_k} [D(w)]_{ls}$.
Applying now Theorem \ref{app_fjm} to the components of $\nabla w$, we obtain
$B\in M^{n\times n}$ so that, in view of (\ref{app_cov}):
\begin{equation}\label{app_proof_3}
\int_\Omega |\nabla w - B|^2 \leq C_{n,L} 
\int_\Omega |\nabla^2 w|^2 \mbox{dist}^2(x,\partial\Omega)~\mbox{d}x
\leq C_{n,L} \int_\Omega |D(w)|^2.
\end{equation}
Define $A=(B- B^T)/2\in so(n)$ and notice that for every $x\in\Omega$ there holds:
\begin{equation*}
\begin{split}
|B-A| = \mbox{dist}_{M^{n\times n}} (B, so(n)) & \leq |B - \nabla w(x)| + 
\mbox{dist}_{M^{n\times n}}(\nabla w(x), so(n)) \\
& = |B - \nabla w(x)| + |D(w)(x)|.
\end{split}
\end{equation*}
Therefore:
\begin{equation}\label{app_proof_4}
\int_\Omega |B-A|^2 \leq C_{n,L} \int_\Omega |D(w)|^2.
\end{equation}
Now by (\ref{app_proof_1}), (\ref{app_proof_3}) and (\ref{app_proof_4}):
\begin{equation}\label{app_proof_5}
\begin{split}
\|\nabla u - A\|_{L^2(\Omega)} & \leq \|\nabla v \|_{L^2(\Omega)} 
+ \|\nabla w - B\|_{L^2(\Omega)} + \|B- A\|_{L^2(\Omega)}\\
& \leq C_{n,L} \left( \|D(u)\|_{L^2(\Omega)} + \|D(w)\|_{L^2(\Omega)} \right) \\
& \leq  C_{n,L} \|D(u)\|_{L^2(\Omega)},
\end{split}
\end{equation}
the last inequality following from $D(w) = D(u) - D(v)$ and the bound (\ref{app_proof_1}).
\endproof

\section{Appendix C - The Killing fields and the Korn inequality on hypersurfaces}
\label{surface_korn}

For a tangent vector field $u\in W^{1,2}(S,\mathbf{R}^n)$, 
define $D(u)$ 
as the symmetric part of its tangential gradient:
$$D(u) = \frac{1}{2} \left[(\nabla u)_{tan} + (\nabla u)_{tan}^T\right].$$
That is, for $x\in S$, $D(u)(x)$ is a symmetric bilinear form given through:
$$\tau^T D(u)(x) \eta = \frac{1}{2} \big(\tau \cdot\partial_{\eta} u(x)
+ \eta\cdot \partial_{\tau} u(x)\big) \qquad \forall \tau,\eta\in T_x S.$$
Recall that a smooth vector field $u$ as above is a Killing field,
provided that $D(u) = 0$ on $S$. We first prove that in presence of this
last condition, the regularity $u\in W^{1,2}(S)$ actually implies that $u$ 
is smooth. Further, we directly recover a generalisation of 
Theorem \ref{korn_standard} (ii) to the non-flat setting 
(Theorem \ref{th_surface_korn}).
Actually, the bound in Theorem \ref{korn_standard} (i)  remains true also in  
the more general framework of Riemannian manifolds \cite{JC}.

\bigskip

The following extension of $u$ on the neighbourhood of $S$ will be useful 
in the sequel:
\begin{equation}\label{def}
\tilde{u}(x+ t\vec{n}(x)) = (\mbox{Id} + t \Pi(x))^{-1} u(x)
\qquad \forall x\in S\quad\forall t\in(-h_0,h_0)
\end{equation}
for some small $h_0>0$. Here $ \Pi(x) = \nabla \vec n(x)$ is the shape operator 
on $S$.  We have $\tilde{u}\in W^{1,2}(\tilde{S},\mathbf{R}^n)$ 
where $\tilde S = S^{h_0}$ is open in $\mathbf{R}^n$. 
Notice that for each $z=x+t\vec n(x)\in \tilde S$ and $\tau_1\in T_x S$
there holds:
\begin{equation*}
\begin{split}
\partial_{\tau_1}\tilde u(z) & = \Big\{\nabla \left[(\mbox{Id} + t \Pi(x))^{-1}\right]
(\mbox{Id} + t \Pi(x))^{-1}\tau_1\Big\} u(x)\\
& \qquad\qquad 
+ (\mbox{Id} + t \Pi(x))^{-1}\nabla u(x)(\mbox{Id} + t \Pi(x))^{-1}\tau_1.
\end{split}
\end{equation*}
The first component above is bounded by $C|t u(x)|$. Taking the scalar product 
of the second component with any $\tau_2\in T_x S$ gives:
$$\big((\mbox{Id} + t \Pi(x))^{-1}\tau_2\big)\cdot \nabla u(x)
(\mbox{Id} + t \Pi(x))^{-1}\tau_1.$$
Since $(\mbox{Id} + t \Pi(x))(T_x S) = T_x S$ we obtain:
\begin{equation}\label{ext_D_1}
\begin{split}
& \tau_2^T D(\tilde u)(z) \tau_1 = \big((\mbox{Id} + t \Pi(x))^{-1}\tau_2\big)\cdot D(u)(x)
(\mbox{Id} + t \Pi(x))^{-1}\tau_1 \\
& \qquad \qquad \qquad + Z(t,x)\cdot u(x),\\
& |Z(t,x)| \leq C.
\end{split}
\end{equation}
On the other hand, $\vec n(x) \cdot\tilde u(z) = 0$, so for any $\tau\in T_x S$:
\begin{equation*}
\begin{split}
\vec n \cdot \partial_{\tau}\tilde u(z) & = - \Big( \Pi(x) (\mbox{Id} + t \Pi(x))^{-1}\tau\Big)
\cdot\tilde u(z) \\
& = - \Big((\mbox{Id} + t \Pi(x))^{-1} \Pi(x) (\mbox{Id} + t \Pi(x))^{-1}u(x)\Big)\cdot\tau
= \tau\cdot\partial_{\vec n}\tilde u(z).
\end{split}
\end{equation*}
Hence:
\begin{equation}\label{ext_D_2}
\begin{split}
& \vec n^T  D(\tilde u)(z) \tau = 
- \Big((\mbox{Id} + t \Pi(x))^{-1} \Pi(x) (\mbox{Id} + t \Pi(x))^{-1}u(x)\Big)\cdot\tau,\\
& \vec n^T D(\tilde u)(z) \vec n = 0.
\end{split}
\end{equation}

\begin{lemma}\label{weak_kil}
Let $u\in W^{1,2}(S,\mathbf{R}^n)$ be a tangent vector field such that $D(u)=0$ almost
everywhere on $S$. Then $u\in\mathcal{I}(S)$.
\end{lemma}
\begin{proof}
We only need to prove that $u$ is smooth. Consider the 
extension $\tilde{u}\in W^{1,2}(\tilde S, \mathbf{R}^n)$ as above.
By (\ref{ext_D_1}), (\ref{ext_D_2}) and the formula (\ref{useful_formula}) we see that
$D(\tilde u)\in W^{1,2}(\tilde S)$ and hence: 
$$\Delta\tilde u \in L^2(\tilde S).$$
The result follows now by the elliptic regularity and a bootstrap argument.
\end{proof}

\begin{theorem}\label{th_surface_korn}
For every tangent vector field $u\in W^{1,2}(S,\mathbf{R}^n)$ there exists $v\in\mathcal{I}(S)$
such that:
$$\|u - v\|_{W^{1,2}(S)} \leq C_S \|D(u)\|_{L^2(S)}$$
and the constant $C_S$ depends only on the surface $S$.
\end{theorem}
\begin{proof}
Since $\mathcal{I}(S)$ is a finitely dimensional subspace of the Banach space $E$
of all $W^{1,2}(S)$
tangent vector fields, its orthogonal complement $\mathcal{I}(S)^\perp$ is a closed subspace 
of $E$. We will prove that:
\begin{equation}\label{surf_uno}
\|u \|_{W^{1,2}(S)} \leq C_S \|D(u)\|_{L^2(S)} 
\qquad \forall w\in \mathcal{I}(S)^\perp
\end{equation}
which implies the Theorem.

If (\ref{surf_uno}) was not true, there would be a sequence $u_n\in \mathcal{I}(S)^\perp$ such that:
$$\|u_n \|_{W^{1,2}(S)} =1, \qquad \|D(u_n)\|_{L^2(S)}\longrightarrow 0 
\quad \mbox{ as } n\longrightarrow \infty$$
Without loss of generality, $u_n$ converge  weakly in $W^{1,2}(S)$ to some $u\in\mathcal{I}(S)^\perp$.
Moreover $D(u)=0$ by the second condition above, so by Lemma \ref{weak_kil} we obtain that
$u\in\mathcal{I}(S)$.

As the spaces $\mathcal{I}(S)$ and $\mathcal{I}(S)^\perp$ are orthogonal, there must be $u=0$, and
hence the sequence   $u_n$ converges to $0$ (strongly) in $L^2(S)$.
This contradicts $\|u_n \|_{W^{1,2}(S)} =1$, because:   
$$\|u_n \|_{W^{1,2}(S)}\leq C_S \left( \|u_n\|_{L^2(S)} 
+ \|D(u_n)\|_{L^2(S)}\right).$$
The last inequality results from Theorem \ref{korn_standard} (i) 
applied to the extensions
$\tilde u_n\in W^{1,2}(\tilde S)$ as in (\ref{def}). 
Indeed, by (\ref{ext_D_1}) and (\ref{ext_D_2}) it follows that:
\begin{equation*}
\begin{split}
 \|\tilde u_n\|_{L^2(\tilde S)} & \approx h_0^{1/2} \|u_n\|_{L^2(S)},\\
 \|\nabla u_n\|_{L^2(S)} & \leq C h_0^{-1/2} \|\tilde u_n\|_{W^{1,2}(\tilde S)},\\
 \|D(\tilde u_n)\|_{L^2(\tilde S)} & \leq C h_0^{1/2} \big(\|u_n\|_{L^2(S)} + \|D(u_n)\|_{L^2(S)}\big).
\end{split}
\end{equation*}
\end{proof}

\bigskip

We now  gather a few remarks relating to the fact 
that the linear space $\mathcal{I}(S)$ of all Killing fields on $S$
is of finite dimension. This is a classical result \cite{KN}, and it implies
that in $\mathcal{I}(S)$ all norms are equivalent. In particular, one has:
\begin{equation}\label{s_uno}
\forall u\in\mathcal{I}(S)\qquad \|\nabla u\|_{L^2(S)} 
\leq C_S \|u\|_{L^2(S)},
\end{equation}
for some constant $C_S$ depending only on the hypersurface $S$.

The bound (\ref{s_uno}), together with an estimate of $C_S$, can also be recovered
directly, using the following identity \cite{Peterson}, 
valid for Killing vector fields $u$:
\begin{equation}\label{s_due}
\Delta_S \left(\frac{1}{2} |u|^2\right) = \left|\widetilde\nabla u\right|^2
- \mbox{Ric } (u,u).
\end{equation}
Here $\Delta_S$ is the Laplace-Beltrami operator on $S$,
$\widetilde\nabla u = (\nabla u)_{tan}$ is the covariant derivative 
of $u$ on $S$, and $\mbox{Ric}$ stands for the Ricci curvature form on $S$.

To calculate $\mbox{Ric }(u,u)$ in our particular setting, notice that
by Gauss' Teorema Egregium (\cite{Spivak}, vol 3),  
the Riemann curvature $4$-tensor on $S$ satisfies:
\begin{equation*}
\forall x\in S\quad \forall \tau,\eta,\xi,\vartheta\in T_x S \qquad
R(\tau,\eta)\xi\cdot\vartheta  = (\Pi(x)\tau\cdot\vartheta) (\Pi(x)\eta\cdot\xi) 
- (\Pi(x)\tau\cdot\xi) (\Pi(x)\eta\cdot\vartheta).
\end{equation*}
Thus, seeing the Ricci curvature $2$-tensor as an appropriate trace of $R$, 
we obtain:
\begin{equation}\label{s_tre}
\begin{split}
\forall x\in S\quad \forall \eta,\xi \in T_x S \qquad
\mbox{Ric } (\eta,\xi) & = \mathrm{tr}~(\tau\mapsto R(\tau,\eta)\xi)\\
& = (\mathrm{tr}~\Pi(x)) \Pi(x)\eta\cdot\xi - \Pi(x)\xi\cdot\Pi(x)\eta\\
& =\left((\mathrm{tr}~\Pi(x))\Pi(x) - \Pi(x)^2\right)\eta\cdot\xi.
\end{split}
\end{equation}
Integrating (\ref{s_due}) on $S$ and using (\ref{s_tre}) we arrive at:
\begin{equation}\label{s_quattro}
\|\widetilde\nabla u\|_{L^2(S)}^2 = \int_S \Big((\mbox{tr }\Pi(x))\Pi(x) -
\Pi(x)^2\Big) u(x)\cdot u(x).
\end{equation}
Notice that in the special case of a $2\times 2$ matrix $\Pi$, that is
when $n=3$ and $S$ is a 2-d surface in $\mathbf{R}^3$, the Cayley-Hamilton
theorem implies:
$$(\mbox{tr }\Pi) \Pi - \Pi^2 =  (\mbox{det }\Pi)\cdot\mbox{Id},$$
and so:
$$\|\widetilde\nabla u\|_{L^2(S)}^2 
= \int_S \mbox{det }\Pi(x) |u|^2.$$
In this case $\mbox{det }\Pi(x)$ is the Gaussian curvature of $S$
at $x$ (see \cite{Peterson}).

To calculate the $L^2$ norm of the full gradient $\nabla u$ on $S$,
notice that:
\begin{equation*}
\|\nabla u\|_{L^2(S)}^2 - \|\widetilde\nabla u\|_{L^2(S)}^2
= \int_S \sum_{i=1}^{n-1}
\left|\vec n\cdot\frac{\partial}{\partial\tau_i} u\right|^2
= \int_S \sum_{i=1}^{n-1} \left|u\cdot \Pi(x)\tau_i\right|^2
= \int_S \left|\Pi(x) u\right|^2.
\end{equation*}
Hence we arrive at:
\begin{equation}\label{s_cinque}
\|\nabla u\|_{L^2(S)}^2 = \int_S (\mbox{tr }\Pi(x)) \Pi(x) u(x)
\cdot u(x),
\end{equation}
which clearly implies (\ref{s_uno}).

\begin{remark}
An equivalent way of obtaining the formula (\ref{s_cinque}), but without using
the language of Riemannian geometry, is to look at 'trivial' extension
of $u$:
$$w(x+t\vec n(x)) = u(x) \qquad \forall x\in S\quad \forall t\in (-h_0,h_0).$$
Since  $\partial_{\vec n} w = 0$ and $w\cdot \vec n = 0$ on the 
boundary of $\tilde S = S^{h_0}$, by (\ref{useful_formula}) one has:
\begin{equation}\label{C_sei}
\|\nabla w\|_{L^2(\tilde S)}^2  = -2 \int_{\tilde S} \mbox{div } D(w)\cdot w 
- \|\mbox{div } w\|_{L^2(\tilde S)}^2.
\end{equation}
Calculating $\int\mbox{div } D(w)\cdot w $ in terms of $\Pi(x)$, 
dividing both sides of (\ref{C_sei}) by 
$2h$ and passing to the limit with $h\longrightarrow 0$, one may recover 
(\ref{s_cinque}) directly.
\end{remark}

\begin{remark}
From the equivalence of the $L^2$ and the $W^{1,2}$ norms on $\mathcal{I}(S)$,
proved in (\ref{s_cinque}), it follows 
that the linear space $\mathcal{I} (S)$ is finitely dimensional.

For otherwise the space $(\mathcal{I}(S), \|\cdot\|_{W^{1,2}(S)})$ 
would have a countable
Hilbertian (orthonormal) base $\{e_i\}_{i=1}^\infty$ and thus necessarily
the sequence $\{e_i\}$ would converge to $0$, weakly in $W^{1,2}(S)$.
But this implies that $\lim_{h\to 0}\|e_i\|_{L^2(S)} = 0$, which
by the norms equivalence gives the same convergence in $W^{1,2}(S)$,
and a contradiction.
\end{remark}

\section{Appendix D - The uniform Poincar\'e inequality and the trace theorem
in thin domains}\label{uniform_Poincare_trace}

\begin{theorem}\label{th_uniform_poincare}
Assume {\bf (H1)} and let $h>0$ be sufficiently small. 
For every $u\in W^{1,2}(S^h, \mathbf{R})$
there exists a constant $a\in\mathbf{R}$ so that:
$$\|u - a\|_{L^2(S^h)}\leq C \|\nabla u \|_{L^2(S^h)}$$
and $C$ is independent of $h$, $a$ or $u$.
\end{theorem}
\begin{proof}
The argument is a combination of the proof of Theorem \ref{approx} 
and the Poincar\'e inequality on fixed surface $S$.
Let $D_{x,h}$, $B_{x,h}$, $\eta_x$ be as in the proof
of Theorem \ref{approx}. Define a smooth function 
$\tilde a:S\longrightarrow \mathbf{R}$:
$$\tilde a(x) =  \int_{S^h} \eta_x (z) u(z)~\mbox{d}z.$$
We will prove the Theorem for $a= \fint_S \tilde{a}(x)~\mbox{d}x.$

First, by Theorem \ref{app_fjm},  the local 
estimate (\ref{1.uno}) can be in our new setting replaced by:
$$\int_{B_{x,h}} |u - a_{x,h}|^2 \leq Ch^2  \int_{B_{x,h}} |\nabla u |^2,$$
with $C$, as usual, a uniform constant.
Repeating the calculations leading to (\ref{1.due}) and (\ref{1.tre}), we thus obtain:
\begin{eqnarray}
&&|\tilde{a}(x) - a_{x,h}|^2 \leq Ch^{2-n} \int_{B_{x,h}} |\nabla u |^2, \nonumber\\
&&|\nabla\tilde{a}(x')|^2 \leq Ch^{-n}  \int_{2B_{x,h}} |\nabla u |^2
\qquad \forall x'\in D_{x,h},\nonumber
\end{eqnarray}
which imply, exactly as in (\ref{1.cinque}):
\begin{equation*}
\int_{S^h} |u-\tilde{a}\pi|^2 \leq Ch^2 \int_{S^h} |\nabla u |^2, 
\qquad
\int_{S} |\nabla\tilde{a}|^2 \leq Ch^{-1} \int_{S^h} |\nabla u |^2. 
\end{equation*}
By the above inequalities and the standard Poincar\'e inequality on surfaces,
it follows:
\begin{equation*}
\begin{split}
\int_{S^h} |u - a|^2 & \leq C \left\{\int_{S^h} |u-\tilde{a}\pi|^2
+ h \int_S |\tilde{a}(x) - a|^2~\mbox{d}x \right\} \\ & \leq 
 C \left\{h^2 \int_{S^h} |\nabla u |^2 + h \int_S|\nabla\tilde{a}|^2 \right\}\leq 
C \int_{S^h}  |\nabla u |^2,
\end{split}
\end{equation*}
proving the result.
\end{proof}

\begin{remark}\label{th_chen_li}
Theorem \ref{th_uniform_poincare} provides a Poincar\'e inequality for sets $\Omega$ 
enjoying properties as in section 10. The following is a more general result.
Assume that $\Omega\subset\mathbf{R}^n$ is open, star-shaped with respect to
the origin and such that:
$$B_r\subset\Omega\subset B_R.$$
Then for every $u\in W^{1,2}(\Omega,\mathbf{R})$ there holds:
$$\left\|u - \fint_\Omega u\right\|_{L^2(\Omega)} \leq C_{n,R/r} R\cdot \|\nabla u\|_{L^2(\Omega)},$$
where the constant $C_{n,R/r}$ depends only on the upper bound of the quantities $n$ and $R/r$.

The proof follows from \cite{CL} where the first nonzero eigenvalue $\alpha_1$
of the Neumann problem for $-\Delta$ on $\Omega$ is estimated from below by 
$C_n\cdot\frac{r^n}{R^{n+2}}$, the constant $C_n$ depending on $n$ only.
Recalling that the best Poincar\'e constant equals to $\alpha_1^{-1/2}$, 
we obtain the result.
\end{remark}

\begin{theorem}\label{th_uniform_trace}
Assume {\bf (H1)}. For every $u\in W^{1,2}(S^h,\mathbf{R})$ there holds:
\begin{equation}\label{trace_uno}
\|u\|_{L^2(S)}\leq Ch^{-1/2} \|u\|_{L^{2}(S^h)} 
+ Ch^{1/2} \|\nabla u\|_{L^{2}(S^h)},
\end{equation}
\begin{equation}\label{trace_due}
\|u\|_{L^2(\partial S^h)}\leq  Ch^{-1/2} \|u\|_{L^{2}(S^h)} 
+ Ch^{1/2} \|\nabla u\|_{L^{2}(S^h)},
\end{equation}
where in the left hand side we have norms of traces of $u$ on $S$ and $\partial S^h$,
respectively. The constant $C$ is independent of $u$ or $h$.
\end{theorem}
\begin{proof}
Since $|g_i^h(x)|\geq Ch$, (\ref{trace_uno}) will be implied by the same inequality 
for $S^h$ with $g_1^h = g_2^h = Ch$. The latter one can be obtained
covering $S^h$ with the cylinders $B_{x,h}$ of size $h$ and applying the scaled version 
of the usual trace theorem to $B_{x,h}$.

Notice, that the constant $C$ in (\ref{trace_uno}) depends only on $n$ 
and the Lipschitz
constant of $S$. Since $|\nabla g_i^h(x)|\leq Ch$ and $|g_i^h(x)|\geq Ch$
for each $s\in S$, we may use the same argument as before on
$\{x - t\vec n^h(x); ~~ x\in \partial S^h, ~ t\in(0,Ch)\}\subset S^h$ 
to prove (\ref{trace_due}).
\end{proof}


\begin{thebibliography}{9999}
\bibitem{CL} R. Chen and P. Li, \textit{On Poincar\'e type inequalities}, 
Trans. AMS, {\bf 349} No.4 (1997), 1561--1585. 

\bibitem{JC} W. Chen and J. Jost, \textit{A Riemannian version of Korn's
inequality}, Calc. Var. {\bf 14}, 517--530. 

\bibitem{ciarbookvol1} P.G. Ciarlet, \textit{Mathematical Elasticity, Vol 1:
Three Dimensional Elasticity}, North-Holland, Amsterdam (1993).


\bibitem{F} K.O. Friedrichs, \textit{On the boundary-value problems 
of the theory of elasticity and Korn's inequality}, Annals of Math. {\bf 48}
No. 2 (1947), 441--471.

\bibitem{FJMgeo} G. Friesecke, R. James and S. M\"uller, 
\textit{A theorem on geometric rigidity and the derivation of nonlinear plate theory from 
three dimensional elasticity}, Comm. Pure. Appl. Math., {\bf 55} (2002), 1461--1506.  

\bibitem{FJMhier} G. Friesecke, R. James, S. M\"uller, \textit{A hierarchy 
of plate models derived from nonlinear elasticity by gamma-convergence}, 
Arch. Ration. Mech. Anal.,  {\bf 180}  (2006),  no. 2, 183--236.

\bibitem{GS} G. Geymonat and P. Suquet,
\textit{Functional spaces for Norton-Hoff materials},
Math. Methods Appl. Sci., {\bf 8} (1986), 206--222.

\bibitem{Grisorods} G. Griso,
\textit{Asymptotic behaviour of curved rods by the unfolding method},
Math. Meth. Appl. Sci, {\bf 27} (2004), 2081--2110.

\bibitem{Grisonone} G. Griso, 
\textit{Asymptotic behavior of structures made of plates}, Analysis and 
Appl., {\bf 3} (2005), 325--356.

\bibitem{Grisonew} G. Griso,
\textit{Decompositions of displacements of thin structures},
J. Math. Pures Appl., {\bf 89} (2008), 199--223.

\bibitem{Horgan} C.O. Horgan, \textit{Korn's inequalities 
and their applications in continuum mechanics},  SIAM Rev.,  {\bf 37}  (1995),  
no. 4, 491--511.

\bibitem{irs}  D. Iftimie, G. Raugel and G.R. Sell
\textit{Navier-Stokes equations in thin 3D domains with Navier boundary
conditions}, to appear in Indiana Math. J.

\bibitem{KN} S. Kobayashi and K. Nomizu, 
\textit{Foundations of Differential Geometry, Vol 1}, Interscience Publishers (1963).

\bibitem{KV2} R.V. Kohn and M. Vogelius,
\textit{A new model for thin plates with rapidly varying thickness.
II:A convergence proof}, Quart. Appl. Math.
{\bf 43} (1985), 1--22.

\bibitem{Korn1} A. Korn, \textit{Solution g\'en\'erale du probl\`eme 
d'\'equilibre dans la th\'eorie de l'\'elasticit\'e dans le cas o\`u les efforts
sont donn\'es \`a la surface}, Ann. Fac. Sci. Toulouse, ser. 2. {\bf 10}
(1908), 165--269.

\bibitem{Korn2} A. Korn, \textit{\"Uber einige Ungleichungen, welche in der 
Theorie der elastischen und elektrischen Schwingungen eine Rolle spielen},
Bull. Int. Cracovie Akademie Umiejet, Classe des Sci. Math. Nat., (1909)
705--724.

\bibitem{kufner} A. Kufner, \textit{Weighted Sobolev Spaces}, 
Wiley and Sons (1985).

\bibitem{OK} V. Kondratiev and O. Oleinik,
\textit{On Korn's inequalities}, C.R. Acad. Sci. Paris, 
{\bf 308} Serie I (1989), 483--487.  

\bibitem{Peterson} P. Petersen,
\textit{Riemannian Geometry}, 2nd edition, Springer (2006).

\bibitem{Ra95} G. Raugel,
\textit {Dynamics of partial differential equations on thin domains},
in CIME Course, Montecatini Terme, Lecture Notes in Mathematics,
{\bf 1609} (1995),  Springer Verlag, 208--315.

\bibitem{RS93} G. Raugel and G.R. Sell,
\textit{Navier-Stokes equations on thin 3D domains. I: Global  attractors 
and global regularity of solutions}, 
J.  Amer. Math. Soc. \textbf{6} (1993),  503--568.

\bibitem{SS} V.A. Solonnikov and V.E. Scadilov,
\textit{A certain boundary value problem for the stationary system of 
Navier-Stokes equations},
Boundary value problems of mathematical physics, 8.
Trudy Mat. Inst. Steklov. {\bf 125} (1973), 196--210, 235. 

\bibitem{Spivak} M. Spivak,
\textit{A Comprehensive Introduction to Differential Geometry}, 
2nd edition, Publish or Perish Inc. (1979).


\end{thebibliography}
\end{document}